\documentclass[12pt,a4]{amsart}
\usepackage[a4paper, left=28mm, right=28mm, top=28mm, bottom=34mm]{geometry}
\usepackage[all]{xy}
\usepackage{amsmath}
\usepackage{amssymb}
\usepackage{amsthm}
\usepackage{comment}
\usepackage{mathrsfs}
\theoremstyle{definition}
\newtheorem{thm}{Theorem}[section]
\newtheorem{lem}[thm]{Lemma}
\newtheorem{cor}[thm]{Corollary}
\newtheorem{prop}[thm]{Proposition}

\theoremstyle{definition}

\newtheorem{rem}[thm]{Remark}
\newtheorem{claim}[thm]{Claim}
\newtheorem{defn}[thm]{Definition}

\newtheorem{ex}[thm]{Example}

\numberwithin{equation}{section}
\def\C{{\mathbb C}}

\def\O{{\mathscr O}}
\def\P{{\mathbb P}}
\def\Q{{\mathbb Q}}
\def\Z{{\mathbb Z}}
\def\Coker{\mathop{\mathrm{Coker}}\nolimits}
\def\Ext{\mathop{\mathrm{Ext}}\nolimits}
\def\Fitt{\mathop{\mathrm{Fitt}}\nolimits}
\def\Frac{\mathop{\mathrm{Frac}}\nolimits}
\def\Hom{\mathop{\mathrm{Hom}}\nolimits}
\def\Ker{\mathop{\mathrm{Ker}}\nolimits}
\def\Spec{\mathop{\rm Spec}}
\def\chara{\mathop{\mathrm{char}}}
\def\idealm{{\mathfrak m}}
\def\idealp{{\mathfrak p}}
\def\jac{\text{\rm jac}}
\def\alg{\text{\rm alg}}
\def\sep{\text{\rm sep}}
\def\length{\mathop{\mathrm{length}}\nolimits}
\def\pr{\text{\rm pr}}

\begin{document}

\title[Deformations of rational curves]{Deformations of rational curves in positive characteristic}

\author{Kazuhiro Ito}
\address{Department of Mathematics, Faculty of Science, Kyoto University, Kyoto 606-8502, Japan}
\email{kito@math.kyoto-u.ac.jp}

\author{Tetsushi Ito}
\address{Department of Mathematics, Faculty of Science, Kyoto University, Kyoto 606-8502, Japan}
\email{tetsushi@math.kyoto-u.ac.jp}

\author{Christian Liedtke}
\address{TU M\"unchen, Zentrum Mathematik - M11, Boltzmannstr.\ 3, 85748
Garching bei M\"unchen, Germany}
\email{liedtke@ma.tum.de}

\date{April 15, 2020}

\maketitle

\begin{abstract}
We study deformations of rational curves and their singularities
in positive characteristic.
We use this to prove that if a smooth and proper surface
in positive characteristic $p$ is dominated by a family of rational curves
such that one member has all $\delta$-invariants (resp.\ Jacobian numbers) 
strictly less than $(p-1)/2$ (resp.\ $p$),
then the surface has negative Kodaira dimension.
We also prove similar, but weaker results hold for higher dimensional varieties.
Moreover, we show by example that our result is in some sense optimal.
On our way, we obtain a sufficient criterion in terms of Jacobian numbers for the normalization of a curve over an imperfect field to be smooth.
\end{abstract}


\section{Introduction}

A \textit{rational curve} is a proper integral scheme
over an algebraically closed field whose normalization
is isomorphic to the projective line $\P^1$.
Rational curves are central to higher dimensional algebraic geometry, as already 
indicated by the title of Koll\'ar's fundamental book \cite{Kollar:RationalCurveBook}.
Let us shortly recall the situation in dimension two.
\begin{enumerate}
\item In characteristic $0$,
it is well-known that
rational curves on surfaces of non-negative Kodaira
dimension are \textit{topologically rigid}, i.e., do not deform in positive-dimensional families.

\item In characteristic $p>0$, then the situation is different:
Zariski gave examples of unirational surfaces of non-negative
Kodaira dimension \cite{Zariski}.
Therefore, rational curves on surfaces of non-negative Kodaira dimension may
\textit{not} be topologically rigid.
However, in this case, the general member of such a positive-dimensional family 
of rational curves is {\em not} smooth.
(In some cases, the singularities of the general member were
studied by Shimada \cite{ShimadaFamily}.)
\end{enumerate}
This poses the interesting question what can be said about topological (non-)rigidity of
rational curves on varieties of non-negative Kodaira dimension in positive characteristics.

We recall some classical invariants of singularities.
Let $C$ be an integral curve over
an algebraically closed field $k$ of characteristic $p>0$.
For each closed point $x\in C$,
the $\delta$-invariant and the Jacobian number of $C$ at $x$
are defined as follows:
\begin{enumerate}
\item The \textit{$\delta$-invariant} is defined by
\[ \delta(C, x) := \dim_{k}(\pi_\ast\O_{\widetilde{C}}/\O_{C})_x, \]
where $\pi \colon \widetilde{C} \to C$ is the normalization morphism.
\item The \textit{Jacobian number} is defined by
\[ \jac(C,x) := \dim_k \left(\O_C/\Fitt^1_{\O_C}(\Omega^1_{C/k})\right)_x, \]
where $\Omega^1_{C/k}$ is the sheaf of K\"ahler differentials on $C$
and $\Fitt^1_{\O_C}(\Omega^1_{C/k}) \subset \O_C$
is the first Fitting ideal of $\Omega^1_{C/k}$.
There are several definitions of Jacobian numbers in the literature.
(See Section \ref{Section:JacobianNumber} for details.)
\end{enumerate}

Next, we recall the notion of families of rational curves
and uniruledness/unirationality of varieties
following \cite{Kollar:RationalCurveBook}.
Let $X$ be a smooth, proper, and connected variety over $k$.
\begin{enumerate}
\item A \textit{family of rational curves on $X$} means
a closed subvariety $\mathscr{C} \subset U \times X$ with projections
$\pi \colon \mathscr{C} \to U$ and $\varphi \colon \mathscr{C} \to X$
such that
$U$ is an integral variety over $k$,
$\pi$ is proper flat, and
every geometric fiber of $\pi$ is an integral rational curve.
We say a rational curve $C \subset X$ is \textit{topologically non-rigid} if
there exists a family of rational curves $(\pi,\varphi)$ on $X$ with
$\dim (\varphi(\mathscr{C})) \geq 2$
such that $\varphi(\mathscr{C}_u) = C$ for some closed point $u \in U$.
Otherwise, we say $C$ is \textit{topologically rigid}.

\item (\cite[Definition IV.1.1.1]{Kollar:RationalCurveBook})
\ We say $X$ is \textit{unirational} if there exists
a dominant rational map
$\psi \colon \P^{\dim(X)} \dashrightarrow X$.
If there exists such a rational map $\psi$ inducing a separable extension of function fields,
we say $X$ is \textit{separably unirational}.
We say $X$ is \textit{uniruled} if there exist an integral variety $Y$ with
$\dim(Y) = \dim(X) - 1$ and a dominant rational map
$\psi \colon \P^1 \times Y \dashrightarrow X$.
If there exists a such a rational map $\psi$ inducing a separable 
extension of function fields,
we say $X$ is \textit{separably uniruled}.
\end{enumerate}


Here is the statement of our main theorem for surfaces.

\begin{thm}
\label{MainTheorem1}
Let $X$ be a smooth, proper, and connected surface over an algebraically closed field $k$
of characteristic $p > 0$.
Assume that $X$ contains a topologically non-rigid rational curve $C \subset X$ satisfying \textit{at least one} of the following conditions:
\begin{enumerate}
\item The $\delta$-invariants of $C$ are strictly less than $(p-1)/2$ at every closed point.
\item The Jacobian numbers of $C$ are strictly less than $p$ at every closed point.
\end{enumerate}
Then, $X$ is separably uniruled and thus, has negative Kodaira dimension.
\end{thm}

\begin{rem}
The converse to Theorem \ref{MainTheorem1} is well-known.
It follows from the classification of surfaces.
In fact, if $X$ is a smooth, proper, and connected surface of negative Kodaira dimension,
then it is birationally equivalent to a ruled surface.
Hence it contains topologically non-rigid {\it smooth} rational curves.
Both the $\delta$-invariant and the Jacobian number of a smooth point are $0$.
\end{rem}


\begin{cor}
\label{Corollary:K3}
Let $X$ be a smooth, proper, and connected surface
of \textit{non-negative} Kodaira dimension over $k$.
Let $C \subset X$ be a rational curve.
\begin{enumerate}
\item If the $\delta$ invariants are strictly less than $(p-1)/2$ at every closed point of $C$, then $C$ is topologically rigid.
\item If the Jacobian numbers are strictly less than $p$ at every closed point of $C$, then $C$ is topologically rigid.
\item If every singularity of $C$ is a node, then $C$ is topologically rigid.
\item If $p \geq 5$ and every singularity of $C$ is either a node or 
  an ordinary cusp, then $C$ is topologically rigid.
\item If $C^2+K_X\cdot C<p-3$, then $C$ is topologically rigid (see Corollary \ref{cor:easy}).
\end{enumerate}
\end{cor}

\begin{rem}
We believe that both invariants are useful: $\delta$-invariants are more often used in the literature and they can be bounded using intersection theory;
see Corollary \ref{cor:easy}.
On the other hand, a node $x\in C$ satisfies $\delta(C,x)=\jac(C,x)=1$, i.e., to conclude the topological rigidity in Corollary \ref{Corollary:K3} (iii) in small characteristics, we have to use the criterion in terms of Jacobian numbers, since the criterion in terms $\delta$-invariants would only give topological rigidity for $p\geq5$.
\end{rem}

\begin{rem}
Our results are optimal in some sense.
Concerning Theorem \ref{MainTheorem1} and the first two parts of Corollary \ref{Corollary:K3},
see Proposition \ref{Proposition:JacobianOptimal} and Example \ref{ex:Quasi-ellipticK3char=3}.
Concerning Corollary \ref{Corollary:K3} (iv),
if $p = 2$ or $3$, then
there exist \textit{quasi-elliptic} surfaces of non-negative Kodaira dimension.
Such surfaces admit a fibration whose general fiber is a rational curve with an ordinary cusp;
see Section \ref{Subsection:QuasiEllipticFibrations}.
\end{rem}


Now, we come to our main results for varieties of higher dimensions.

\begin{thm}
\label{MainTheorem2}
Let $X$ be a smooth, proper, and connected variety over
an algebraically closed field $k$ of characteristic $p>0$
with $\dim(X) \geq 2$.
Assume that there exist an integral variety $U$ with
$\dim(U) = \dim(X) - 1$ and a closed subvariety 
$\mathscr{C} \subset U \times X$ with projections $\pi \colon \mathscr{C} \to U$ 
and $\varphi \colon \mathscr{C} \to X$ such that
\begin{enumerate}
\item $\varphi \colon \mathscr{C} \to X$ is dominant,
\item $\mathscr{C}$ gives rise to a family of rational curves on $X$, and 
\item $k(\mathscr{C}) \cap k(X)^{\sep}$ is a separable extension of $k(U) \cap k(X)^{\sep}$.
\end{enumerate}
Moreover, we assume that there is a closed point $u_0 \in U$ which satisfies \textit{at least one} of the following conditions:
\begin{enumerate}
\setcounter{enumi}{3}
    \item The $\delta$-invariants of $\mathscr{C}_{u_0}$ are strictly less than $(p-1)/2$ at every closed point.
    \item $\mathscr{C}_{u_0}$ is a local complete intersection rational curve on $X$ whose Jacobian numbers are strictly less than $p$ at every closed point.
\end{enumerate}
Then, $X$ is separably uniruled and thus, has negative Kodaira dimension.
\end{thm}

Concerning the intersection in (iii), we let 
$k(X)^{\sep}$ be the separable closure of $k(X)$ in a fixed 
algebraic closure $k(X)^{\alg}$ of $k(X)$.
The generically finite morphism $\varphi$ induces a finite field extension
$k(X)\subset k(\mathscr{C})$ and we may embed $k(\mathscr{C})$ into
$k(X)^{\rm alg}$.
Thus, $k(\mathscr{C}) \cap k(X)^{\sep}$ is the separable closure of $k(X)$
inside $k(\mathscr{C})$.
We embed $k(U)$ into $k(\mathscr{C})$ via $\pi$.

\begin{rem}
By a theorem of Mac Lane \cite{MacLane},
a field extension $L/K$ in characteristic $p>0$ is separable if and only
if $L\otimes_KK^{1/p}$ is a field.
\end{rem}



\begin{rem}
If $\dim(X) \geq 3$, the condition (iii) is really necessary;
see Proposition \ref{Proposition:CounterexampleHigherDimension}.
It is related to the characteristic-$p$ phenomenon
of the existence of fibrations between smooth varieties,
whose geometric generic fiber is not reduced as in \cite{MoriSaito, Sato:Uniruled, Schroeer:Nagoya}.
\end{rem}

\begin{rem}
In the course of the proofs of main results of this article,
we give a sufficient criterion in terms of Jacobian numbers
(resp.\ $\delta$-invariants) for the smoothness of the normalization of a curve over an imperfect field of characteristic $p > 0$.
Such results might be of independent interest.
See Theorem \ref{MainTheorem3} and Theorem \ref{Theorem:TateGenusChange} for details.
\end{rem}

This article is organized as follows.
In Section \ref{Section:KodairaDimension},
we recall some well-known results about the Kodaira dimension, about
separably uniruled varieties, and about birationally ruled varieties.
In Section \ref{Section:FamilyRationalCurve},
we discuss rational curves, maps from curves with rational components,
and topological rigidity.
In Section \ref{Section:JacobianNumber}, we collect some basic properties of
Jacobian numbers of curves over arbitrary fields.
In Section \ref{Section:DeltaInvariant}, we recall the definition and basic properties of $\delta$-invariants of curves over arbitrary fields.
In Section \ref{Section:KeyLemma},
we prove a lemma that is used in the proof of
Theorem \ref{MainTheorem1} and Theorem \ref{MainTheorem2}.
In Section \ref{Section:ProofMainTheorem},
we prove Theorem \ref{MainTheorem2}.
Then, Theorem \ref{MainTheorem1} is an easy consequence of 
Theorem \ref{MainTheorem2}.\
Finally, in Section \ref{Section:Examples},
we give some examples of topologically rigid and topologically non-rigid rational curves on surfaces.
We also give counterexamples to a naive generalization of
Theorem \ref{MainTheorem1} to higher dimensions.

\section{The Kodaira dimension of separably uniruled varieties}
\label{Section:KodairaDimension}

Let $k$ be an arbitrary field. 
In this article, for simplicity, a \textit{variety} over $k$ means
a separated scheme of finite type over $k$.
A \textit{surface} over $k$ means a variety of pure dimension $2$ over $k$.
Moreover, a \textit{curve} over $k$ means a variety of pure dimension $1$ over $k$.

The following result is well-known.

\begin{prop}
\label{Proposition:SeparablyUniruledPlurigenera}
Let $X$ be a smooth, proper,
and connected variety over an algebraically closed field $k$.
\begin{enumerate}
\item If $X$ is separably uniruled, then $X$ has negative Kodaira dimension.
\item Moreover, if $X$ is a surface,
then the following are equivalent:
\begin{enumerate}
\item $X$ is birationally equivalent to a ruled surface.
\item $X$ is separably uniruled.
\item $X$ has negative Kodaira dimension.
\end{enumerate}
\end{enumerate}
\end{prop}

\begin{proof}
(i) If $X$ is separably uniruled, then $H^0(X,K_X^{\otimes m}) = 0$ for every $m \geq 1$; 
see \cite[Corollary IV.1.11]{Kollar:RationalCurveBook}.
Hence, $X$ has negative Kodaira dimension.

(ii) This follows from the classification of surfaces; see e.g.\ \cite[Theorem 13.2]{Badescu}.
\end{proof}

\begin{rem}
It is conjectured that the converse to Proposition \ref{Proposition:SeparablyUniruledPlurigenera} (i) 
holds (at least in characteristic zero); see \cite[Conjecture IV.1.12]{Kollar:RationalCurveBook} for details
and partial results in this direction.
\end{rem}

\begin{lem}
\label{Lemma:SeparablyUniruledFiberProduct}
Let $X$ be a smooth, proper, and connected variety
over an algebraically closed field $k$.
Let $C$ be a smooth, proper, and connected curve over $k$ 
of genus $g \geq 1$.
Then, $X$ is separably uniruled
if and only if $X \times C$ is separably uniruled.
\end{lem}

\begin{proof}
Obviously, if $X$ is separably uniruled, then
$X \times C$ is separably uniruled.
Conversely, assume that $X \times C$ is separably uniruled.
By \cite[Theorem IV.1.9]{Kollar:RationalCurveBook}, there exists 
a free morphism $f \colon \P^1 \to X \times C$.
Namely, $H^1(\P^1, f^{*}T_{X \times C})=0$ and 
$f^{*}T_{X \times C}$ is generated by global sections,
where $T_{X \times C}$ denotes the tangent sheaf of $X \times C$.
Since the genus $g$ of $C$ is strictly larger than $0$, 
the image of the morphism $\P^1 \to C$,
which is the composition of $f$ and the projection onto
the second factor
$\mathrm{pr}_2 \colon X \times C \to C$
is a closed point.
Hence $f$ factors as $\P^1 \to X \to X \times C$.
Since $f$ is a free morphism, 
it follows that $\P^1 \to X$ is also free.
Using \cite[Theorem IV.1.9]{Kollar:RationalCurveBook} again,
we conclude that $X$ is separably uniruled.
\end{proof}

\section{Families of rational curves and maps from curves with rational components.}
\label{Section:FamilyRationalCurve}

In this section, we fix some definitions and recall some basic properties of
topologically non-rigid rational curves and maps from curves with rational components.
The standard reference for rational curves on varieties
is Koll\'ar's book \cite{Kollar:RationalCurveBook}.
The following definitions of rational curves and maps from curves with rational components will be used in the present article.
(Slightly different notions might be used in the literature, but
there should be no major differences.)

\begin{defn}
\label{Definition:FamilyRationalCurve}
Let $X$ be a proper variety over an algebraically closed field $k$.
Let $U$ be an integral variety over $k$.
\begin{enumerate}
\item
A {\it rational curve} on $X$ is an integral closed subvariety
$C \subset X$ of dimension $1$, whose normalization is isomorphic
to $\P^1$ over $k$.

\item
A \textit{flat family of rational curves} parameterized by $U$
is a proper flat morphism $\pi \colon \mathscr{C} \to U$ such that,
for every geometric point $s \to U$,
the geometric fiber $\mathscr{C}_{s} := s \times_{U}{\mathscr{C}}$
is an integral rational curve over the residue field $\kappa(s)$.
More generally,
a \textit{flat family of curves with rational components} parameterized by $U$
is a proper flat morphism $\pi \colon \mathscr{C} \to U$ such that,
for every geometric point $s \to U$,
the geometric fiber $\mathscr{C}_{s}$
is a reduced curve over $\kappa(s)$ whose irreducible components are rational curves.

\item
A \textit{family of rational curves on $X$} parametrized by $U$ is
a closed subscheme $\mathscr{C} \subset U \times X$ such that
the projection
$\pi \colon \mathscr{C} \to U$
is a flat family of rational curves.
Let $\varphi \colon \mathscr{C} \to X$ be the projection to $X$.

A rational curve $C \subset X$ is said to be \textit{topologically non-rigid} if
there exists a family of rational curves $(\pi,\varphi)$ on $X$ with
$\dim (\varphi(\mathscr{C})) \geq 2$ and
$\varphi(\mathscr{C}_u) = C$
for some closed point $u \in U$.
Otherwise, we say $C$ is \textit{topologically rigid}.
(In particular, a topologically rigid curve in our sense is allowed to deform 
infinitesimally on $X$, but not in a positive dimensional family.)

\item 
A \textit{map from a family of rational curves} to $X$
parametrized by $U$ is a pair of morphisms
$\pi \colon \mathscr{C} \to U$ and $\varphi \colon \mathscr{C} \to X$ over $k$
such that
\begin{enumerate}
\item $\pi$ is a flat family of rational curves, and
\item $\dim(\varphi(\mathscr{C}_{s}))=1$
for every geometric point $s \to U$.
\end{enumerate}

\item
A \textit{map from a curve with rational components} to $X$ is a morphism
$f \colon C \to X$ over $k$ such that
\begin{enumerate}
\item $C$ is a reduced curve over $k$ whose irreducible components are rational curves, and
\item the image $f(C)$ is of pure dimension $1$.
\end{enumerate}
We say $f \colon C \to X$ is a \textit{generic immersion}
if the restriction of $f$ to an open dense subset of $C$ is
an immersion.

\item
A \textit{map from a family of curves with rational components} to $X$
parametrized by $U$ is a pair of morphisms
$\pi \colon \mathscr{C} \to U$ and $\varphi \colon \mathscr{C} \to X$ over $k$
such that
\begin{enumerate}
\item $\pi$ is a flat family of curves with rational components, and
\item for every geometric point $s \to U$, the induced map
$\varphi_{s} \colon \mathscr{C}_{s} \to X_{s}$
is a map from a curve with rational components to $X_s$.
\end{enumerate}

\item A map from a curve with rational components $f \colon C \to X$ over $k$ is \textit{topologically non-rigid}
if there exists a pair $(\pi,\varphi)$ as in (vi) such that
\begin{enumerate}
\item $\dim \varphi(\mathscr{C}) \geq 2$, and
\item $\varphi_{u_0} \colon \mathscr{C}_{u_0} \to X$
is identified with $f$ for some closed point $u_0 \in U$.
\end{enumerate}
If there does not exist such a pair $(\pi,\varphi)$, then
we say that $f$ is \textit{topologically rigid}.
\end{enumerate}
\end{defn}


\begin{lem}
\label{Lemma:StableMapFamilyImage}
Let $X$ be a proper and integral variety over an algebraically closed field $k$,
and let $U$ be an integral variety over $k$.
Let $\pi \colon \mathscr{C} \to U$ and $\varphi\colon\mathscr{C}\to X$ be morphisms over $k$.
Assume that $\mathscr{C}$ is reduced and that $\pi$ is proper and flat with one-dimensional 
fibers.
Let
$ W := (\pi \times \varphi)(\mathscr{C})$
be the image of
$\pi \times \varphi \colon \mathscr{C} \to U \times X$
endowed with the reduced induced subscheme structure.
Let $\pr_1 \colon W \to U$ be the projection onto the first factor.
\begin{enumerate}
\item Assume that the fiber $\mathscr{C}_{u}$ is reduced for some closed point $u \in U$.
Then, there exists an open dense subset $U' \subset U$ such that
the fiber
$ \pr_1^{-1}(s) := s \times_{U} W$
is reduced for every geometric point $s \to U'$.

\item 
Assume moreover that $X$ is a smooth, proper, and connected surface, $U$ is a smooth curve, $\mathscr{C}$ is irreducible, and the fiber $\mathscr{C}_{u_0}$ is generically reduced
and $\varphi_{u_0} \colon \mathscr{C}_{u_0} \to X$ is a generic immersion for some closed point $u_0 \in U$.
Then, there exists an open neighborhood
$u_0 \in U' \subset U$
such that the fiber
$\pr_1^{-1}(s)$
is reduced for every geometric point $s \to U'$.
\end{enumerate}
\end{lem}

\begin{proof}
(i) Let $\overline{\eta} \to U$ be the geometric generic point. 
The fiber $\pr_1^{-1}(\overline{\eta})$ is the schematic image of the morphism
$\mathscr{C}_{\overline{\eta}} \to X \otimes_{k} \kappa(\overline{\eta})$. 
Since $\mathscr{C}_{\overline{\eta}}$ is reduced by \cite[Th\'eor\`eme 12.2.4 (v)]{EGA4-3},
it follows that $\pr_1^{-1}(\overline{\eta})$ is reduced. 
After possibly shrinking $U$, the fiber $\pr_1^{-1}(s)$ is reduced for every 
geometric point $s \to U$ by \cite[Th\'eor\`eme 12.2.4 (v)]{EGA4-3}.

(ii)  It is enough to show that the fiber $\pr_1^{-1}(u_0)$ is reduced; see \cite[Th\'eor\`eme 12.2.4 (v)]{EGA4-3}.
Since $U \times X$ is smooth over $k$, its reduced closed subscheme $W \subset U \times X$ 
is a Cartier divisor.
Moreover $W$ is flat over $U$ since $U$ is a smooth curve over $k$. 
Since $\pr_1^{-1}(u_0)$ has no embedded points by \cite[Chapter 8, Proposition 2.15]{Liu:Book},
we only need to prove that it is generically reduced; see \cite[Chapter 7, Exercise 1.2]{Liu:Book}.
We consider $\mathscr{C}_{u_0} = \pi^{-1}(u_0) \subset \mathscr{C}$ and
$\pr_1^{-1}(u_0) \subset W$ as Cartier divisors on $\mathscr{C}$ and $W$, respectively.
Since $\mathscr{C}_{u_0} = (\pi \times \varphi)^{\ast} \pr_1^{-1}(u_0)$,
we have the following equality of $1$-cycles on $W$:
\[
(\pi \times \varphi)_{\ast} [\mathscr{C}_{u_0}]
\,=\, (\pi \times \varphi)_{\ast} \left((\pi \times \varphi)^{\ast} [\pr_1^{-1}(u_0)]\right)
\,=\, d\cdot [\pr_1^{-1}(u_0)],
\]
where $d := [k(\mathscr{C}) : k(W)]$ is the extension degree of function fields; see \cite[Theorem 7.2.18]{Liu:Book}.
By our assumptions on $\varphi_{u_0}$, 
we have $d = 1$ and thus, the Cartier divisor $\pr_1^{-1}(u_0) \subset W$
has multiplicity one.
Consequently, the fiber $\pr_1^{-1}(u_0)$ is generically reduced.
\end{proof}

The following result is well-known, at least in characteristic zero. 
We give a brief sketch of the proof for the reader's convenience. 
(See also \cite[Proposition IV.1.3]{Kollar:RationalCurveBook}.)

\begin{prop}
\label{Proposition:UniruledExistenceNonRigidRationalCurves}
For a proper and integral variety $X$ with $\dim(X) \geq 2$
over an algebraically closed field $k$,
the following conditions are equivalent:
\begin{enumerate}
\item $X$ is uniruled.

\item $X$ is dominated by a family of rational curves on $X$, i.e.,
there exist an integral variety $U$ with $\dim(U) = \dim(X) - 1$
and a closed subvariety $\mathscr{C} \subset U \times X$
as in Definition \ref{Definition:FamilyRationalCurve} (iii)
such that $\varphi \colon \mathscr{C} \to X$ is dominant.

\item $X$ is dominated by a family of curves with rational components, i.e.,
there exist an integral variety $U$ with $\dim(U) = \dim(X) - 1$
and a pair $(\pi,\varphi)$
as in Definition \ref{Definition:FamilyRationalCurve} (vi)
such that $\varphi \colon \mathscr{C} \to X$ is dominant.
\end{enumerate}
\end{prop}

\begin{proof}
(i) $\Rightarrow$ (iii):
Assume that $X$ is uniruled. Then there exists a dominant rational map
$\psi \colon \P^1 \times Y \dashrightarrow X$
with $\dim(Y) = \dim(X) - 1$.
Shrinking $Y$ if necessary, we may assume $Y$ is smooth.
Then, $\psi$ is defined in codimension $1$ and thus,
there exists a closed subvariety $Z \subset \P^1 \times Y$ with
$\dim(Z) \leq \dim(Y) - 1 $
such that $\psi$ is defined outside $Z$.
Removing $\pr_2(Z)$ from $Y$, we may assume that $\psi$ is defined everywhere.
Then, $\psi \colon \P^1 \times Y \to X$ gives rise to
a map from a family of curves with rational components parametrized by $Y$
and dominating $X$.

(iii) $\Rightarrow$ (ii):
Take a pair $(\pi,\varphi)$ as in
Definition \ref{Definition:FamilyRationalCurve} (vi).
Replacing $\mathscr{C}$ by an irreducible component that dominates $X$ and shrinking $U$, 
we may assume that the fiber $\pr_1^{-1}(s)$ of the image
$ W := (\pi \times \varphi)(\mathscr{C}) $
is an integral rational curve 
for every geometric point $s \to U$ by Lemma \ref{Lemma:StableMapFamilyImage} (i).

(ii) $\Rightarrow$ (i):
Choose a closed subvariety $\mathscr{C} \subset U \times X$
as in Definition \ref{Definition:FamilyRationalCurve} (iii).
Let $K := k(U)$ be the function field of $U$.
After replacing $U$ by a finite covering $U' \to U$ and
replacing $\mathscr{C}$ by the normalization of the base change
$\mathscr{C} \times_U U'$,
we find a dominant morphism $\mathscr{C} \to X$
such that the generic fiber $\mathscr{C}_K$ is
a geometrically irreducible and smooth curve over $K$; 
see \cite[Proposition 17.15.14]{EGA4-4}.
Moreover, shrinking $U$ further and replacing $U$ by an \'etale covering,
we may assume that $\mathscr{C} \to U$ is a $\P^1$-bundle.
Hence, $X$ is uniruled.
\end{proof}

\section{Jacobian numbers of curves over arbitrary fields}
\label{Section:JacobianNumber}


In this section, we fix an arbitrary field $k$ of characteristic $p\geq0$
and we recall the definition and basic properties of Jacobian numbers of curves over $k$.
Most of the results in this section are well-known if $k = \C$.
We also give brief proofs of the results recalled in this section
because we need to apply them to curves over function fields of curves,
for which we could not find appropriate references.

\subsection{Definition of Jacobian numbers}

Let $k^{\alg}$ be an algebraic closure of $k$ and let
$k^{\rm{sep}}$ be the separable closure of $k$ in $k^{\alg}$.
Let $C$ be a curve over $k$. 
(Note that in this article, a \textit{curve} over $k$ is possibly non-proper, non-reduced, or reducible.)
Let $\Omega^1_{C/k}$ be the sheaf of K\"ahler differentials on $C$
and let $\Fitt^1_{\O_C}(\Omega^1_{C/k}) \subset \O_C$ be
the first Fitting ideal of $\Omega^1_{C/k}$.
(For the definition and basic properties of Fitting ideals,
we refer to \cite[Section 16.29]{GoertzWedhorn}.)

\begin{defn}
\label{Definition:JacobianNumber}
For a closed point $x \in C$,
the \textit{Jacobian number} of $C$ at $x$ is defined by
$\jac(C,x) := \dim_k \left(\O_C/\Fitt^1_{\O_C}(\Omega^1_{C/k})\right)_x$.
\end{defn}

\begin{rem}
There are several definitions of Jacobian numbers in the literature.
In this article, we adopt Schr\"oer's definition in terms of Fitting ideals; see
\cite[Section 3, p.~64]{Schroeer}.
The advantage of this definition is that it makes sense for all curves
and that it behaves well scheme-theoretically.
For plane curves, it coincides with the more traditional definition
of Jacobian numbers (as in \cite{BuchweitzGreuel}, \cite{GreuelLossenShustin}, or \cite{Tjurina})
in terms of the dimension of $\Ext^1$ of sheaves,
or in terms of the ideal generated by partial derivatives of the defining equation; see
Proposition \ref{Proposition:lci} and
Corollary \ref{Corollary:JacobianNumberDerivative} for details.
\end{rem}

\begin{prop}
\label{Proposition:JacobianNumberZeroSmoothness}
Let $C$ be a curve over $k$.
Then, for a closed point $x \in C$,
we have $\jac(C,x) = 0$ if and only if $C$ is smooth at $x$.
\end{prop}

\begin{proof}
We have $\jac(C,x) = 0$ if and only if  $\Omega^1_{C/k}$ is locally free of rank $1$ 
in an open neighborhood of $x$; see \cite[Remark 16.30]{GoertzWedhorn}.
Hence, by \cite[Chapter 6, Proposition 2.2]{Liu:Book},
we have $\jac(C,x) = 0$ if and only if $C$ is smooth at $x$.
\end{proof}

The closed subscheme of $C$ defined by
$\Fitt^1_{\O_C}(\Omega^1_{C/k})$ is called the \textit{Jacobian subscheme}
and by the previous proposition, its support coincides with the non-smooth locus of 
$C$ over $k$.
We say a curve $C$ over $k$ is \textit{geometrically reduced}
if $C \otimes_k k^{\alg}$ is reduced.
If $C$ is geometrically reduced, then the smooth locus of $C$ over $k$ 
is open and dense. 
It follows that we have $\jac(C,x) = 0$ for all but finitely 
many closed points $x \in C$ and that we have $\jac(C,x) < \infty$ 
for every closed point $x \in C$.

\subsection{Regular curves over imperfect fields with small Jacobian numbers}

\begin{prop}
\label{Proposition:SmallJacobianNumberSmoothness}
Let $k$ be a field of characteristic $p>0$ that is not necessarily perfect.
Let $C$ be a curve over $k$.
Assume that the Jacobian numbers of $C$ are strictly less than $p$ at every closed point of $C$,
and $C$ is a regular scheme.
Then, $C$ is smooth over $k$.
\end{prop}

\begin{proof}
Seeking a contradiction, assume that there exists a closed point $x \in C$ such that
$C \to \Spec k$ is not smooth around $x$.
Since $C$ was assumed to be regular, it follows that the residue field extension 
$\kappa(x)/k$ is not separable; see \cite[Proposition 3.2]{Schroeer}.
In particular, $[\kappa(x) : k]$ is divisible by $p$.
Hence the dimension of the stalk $(\O_C/\Fitt^1_{\O_C}(\Omega^1_{C/k}))_x$ as a $k$-vector space is divisible by $p$.
(See also \cite[Lemma 3.3 and Proposition 3.6]{Schroeer}.)
On the other hand, since $\jac(C,x) < p$, we have $\jac(C,x) = 0$.
Since $C$ is not smooth around $x$, this contradicts 
Proposition \ref{Proposition:JacobianNumberZeroSmoothness}.
\end{proof}

\subsection{Closed subschemes and finite base change}

\begin{prop}
\label{Proposition:JacobianNumberClosedSubscheme}
Let $C$ and $C'$ be curves over a field $k$ together with a closed immersion 
$i \colon C' \hookrightarrow C$.
For every closed point $x \in C'$, we have
$\jac(C',x) \leq \jac(C,x).$
\end{prop}

\begin{proof}
Since $i$ is a closed immersion,
we have a natural surjection $i^{\ast} \Omega^1_{C/k} \to \Omega^1_{C'/k}$;
see \cite[Chapter 6, Proposition 1.24(d)]{Liu:Book}.
Hence, we have
\[
i^{\ast} \Fitt^1_{\O_C}(\Omega^1_{C/k})
\,=\, \Fitt^1_{\O_{C'}}(i^{\ast} \Omega^1_{C/k})
\,\subset\, \Fitt^1_{\O_{C'}}(\Omega^1_{C'/k}),
\]
from which the assertion follows.
\end{proof}

\begin{prop}
\label{Proposition:JacobianNumberEtaleCovering}
Let $C$ be a curve over a field $k$ and let $k'/k$ be a field extension.
We set $C_{k'} := C \otimes_k k'$ and denote by $p \colon C_{k'} \to C$
the natural morphism.
For every closed point $x \in C$, we have
$\jac(C,x) = \sum_{y \in p^{-1}(x)} \jac(C_{k'},y)$.
\end{prop}

\begin{proof}
Since $p^{\ast} \Omega^1_{C/k} = \Omega^1_{C_{k'}/k'}$,
we have
\[ p^{\ast} \Fitt^1_{\O_C}(\Omega^1_{C/k}) = \Fitt^1_{\O_{C_{k'}}}(\Omega^1_{C_{k'}/k'}); \]
see \cite[Proposition 16.29 (3)]{GoertzWedhorn}.
The assertion easily follows from this equality.
\end{proof}

\subsection{Local complete intersection curves}

In this subsection, we study Jacobian numbers of
curves that are \textit{local complete intersections};
for the definition of this notion, see \cite[Chapter 6, Definition 3.17]{Liu:Book}. 
For a local complete intersection curve $C$,
we can calculate Jacobian numbers using dualizing sheaves:
we have a canonical map
$c_{C/k} \colon \Omega^1_{C/k} \to \omega_{C/k},$
called the \textit{class map},
from the sheaf of K\"ahler differentials $\Omega^1_{C/k}$
to the dualizing sheaf $\omega_{C/k}$;
see \cite[Chapter 6, Corollary 4.13]{Liu:Book}.

\begin{prop}
\label{Proposition:lci}
Let $C$ be a local complete intersection curve over $k$.
For a closed point $x \in C$, we have
$ \jac(C,x) = \dim_k \Coker (c_{C/k})_x$.
Moreover, if $C$ is geometrically reduced,
then we have 
\[ \jac(C,x) \,=\, \dim_k \Coker (c_{C/k})_x \,=\, \dim_k \Ext^1_{\O_{C,x}}(\Omega^1_{C/k,x}, \O_{C,x}). \]
\end{prop}

\begin{proof}
The first equality follows from an explicit description of the class map $c_{C/k}$
as in \cite[Chapter 6, Section 6.4.2]{Liu:Book}.
Let us briefly recall it.
Since $C$ is a local complete intersection over $k$,
there exists an affine open neighborhood $U$ of $x \in C$
such that
$
U \cong \Spec A
$
, where
$A = k[T_1,\ldots,T_{n+1}]/I
$
for an ideal $I = (F_1,\ldots,F_n)$ of $k[T_1,\ldots,T_{n+1}]$.
The dualizing module $\omega_{A/k} := \Gamma(U,\omega_{U/k})$ is
given by
\[ 
\det (I/I^2)^{\vee} \otimes_A \left( \det (\Omega^1_{k[T_1,\ldots,T_{n+1}]/k}) \otimes_{k[T_1,\ldots,T_{n+1}]} A\right). 
\]
It is a free $A$-module of rank one
with basis
\[
e \,:=\, (\overline{F}_1 \wedge \cdots \wedge \overline{F}_n)^{\vee}
     \otimes \left((d T_1 \wedge \cdots \wedge d T_{n+1}) \otimes 1_A\right),
\]
where $\overline{F}_i$ is the image of $F_i$ in $I/I^2$
and where $1_A \in A$ is the identity; see \cite[Chapter 6, Lemma 4.12]{Liu:Book}.
With respect to this basis, the class map $c_{U/k}$ is given by
\[
c_{U/k} \colon \Omega^1_{A/k} \to \omega_{A/k}, \quad
dt_i \mapsto \Delta_i\cdot e. \]
Here, $dt_i \in \Omega^1_{A/k}$ is the image of $T_i$ in $\Omega^1_{A/k}$
under the universal derivation and $\Delta _i \in A$ denotes 
the determinant of the Jacobian matrix $(\partial F_i/\partial T_j)_{i,j}$ with 
$i$.th column removed.
Therefore, under the isomorphism $A \cong \omega_{A/k}$ that sends $1_A$ to $e$,
the ideal of $A$ corresponding to the image of the class map $c_{U/k}$
is equal to the first Fitting ideal $\Fitt^1_A(\Omega^1_{A/k})$;
see \cite[Section 16.9]{GoertzWedhorn}.
This establishes the first equality.

The second equality was essentially proved by Rim in \cite{Rim}.
However, there it is somewhat implicit in the proofs of the main theorems of \cite{Rim},
as well as under the additional assumption that $k$ is 
perfect.
Let us briefly explain how to deduce the second equality from
the results in \cite{Rim}:
since the statement is local, we may use the same setup and notation as before.
Shrinking $U$ if necessary, we may assume $U \backslash \{ x \}$ is smooth over $k$.
Let $(\Omega^1_{A/k})_{\rm{tors}}$ be the torsion submodule of
$\Omega^1_{A/k}$, i.e.,
\[
(\Omega^1_{A/k})_{\rm{tors}} \,:=\, 
\left\{ \, m \in \Omega^1_{A/k} \, \vert \, {\rm{there\, is \, a\, regular\, element}}\, a \in A\, {\rm{such\, that }}\, am=0 \, \right\}.
\]
Rim proved the equality
\[ \length_A (\Omega^1_{A/k})_{\rm{tors}} = \length_A (\Coker c_{U/k}) \]
using the generalized Koszul complexes of Buchsbaum and Rim;
see \cite[Theorem 1.2 (ii)]{Rim} and \cite[Corollary 1.3 (ii)]{Rim}.
(In fact, \cite[Theorem 1.2 (ii)]{Rim} is a general result in commutative algebra,
which is valid for Cohen-Macaulay algebras.
The perfectness of the base field was not used there.)
Let $\idealm_x \subset A$ be the maximal ideal corresponding to $x \in C$
and let $A_x$ be the localization of $A$.
Since $U \backslash \{ x \}$ is smooth over $k$,
we have 
$(\Omega^1_{A/k})_{\rm{tors}} = (\Omega^1_{A_x/k})_{\rm{tors}}$,
which is an $A_x$-module of finite length.
Since $A_x$ is a one-dimensional Gorenstein local ring,
Grothendieck's local duality gives an isomorphism
\[ 
\Ext^1_{A_x}\left(\Omega^1_{A_x/k}, \omega_{A_x/k}\right) \,\cong\, 
\Hom_{A_x}\left(H^0_{\idealm_x}(\Omega^1_{A_x/k}), E(A_x/\idealm_x)\right), 
\]
where $E(A_x/\idealm_x)$ is an injective hull of the residue field $A_x/\idealm_x$; 
see \cite[Theorem 3.5.8]{BrunsHerzog}.
Since $A_x$ is a Cohen-Macaulay ring and $(\Omega^1_{A_x/k})_{\rm{tors}}$ 
is a module of finite length, we have
\[ H^0_{\idealm_x}(\Omega^1_{A_x/k}) = (\Omega^1_{A_x/k})_{\rm{tors}}. \]
The right hand side of the above isomorphism is identified with
the Matlis dual of $(\Omega^1_{A_x/k})_{\rm{tors}}$.
Since Matlis duality preserves the length of Artinian modules
(see \cite[Proposition 3.2.12]{BrunsHerzog}), we have
\[ \length_{A_x} \Ext^1_{A_x}(\Omega^1_{A_x/k}, \omega_{A_x/k}) 
= \length_{A_x} (\Omega^1_{A_x/k})_{\rm{tors}}. \]
Since $\omega_{A_x/k}$ is a free $A_x$-module of rank one,
the second equality of this proposition follows from above results.
\end{proof}

\begin{rem}
Proposition \ref{Proposition:lci} is well-known if $C$ is a plane curve over 
the complex numbers; see,
for example \cite[Lemma 1.1.2, Corollary 6.1.6]{BuchweitzGreuel} or
\cite[Chapter II, p.~317, the proof of Lemma 2.32]{GreuelLossenShustin}.
The second equality in Proposition \ref{Proposition:lci}
was attributed to Rim in \cite[Proposition 2.2]{EstevesKleiman}.
We refer to the proof of \cite[Proposition 2.2]{EstevesKleiman}
for a historical account.
\end{rem}


\begin{prop}
\label{Proposition:JacobianNumberFiniteMorphism}
Let $C$ and $C'$ be two reduced local complete intersection curves over $k$.
Let $f \colon C' \to C$ be a finite morphism over $k$ and assume that there
exists an open dense subset $U \subset C$ such that the restriction
$f|_{f^{-1}(U)} \colon f^{-1}(U) \to U$
is an isomorphism.
Let $g$ be the composition
$\Omega^1_{C/k} \to
f_{\ast} \Omega^1_{C'/k} \to
f_{\ast} \big( \Omega^1_{C'/k}/\Ker(c_{C'/k}) \big)$.
If $x \in C$ is a closed point, then 
\[
\jac(C,x) 
\,=\,
\dim_k \Coker(g)_x
\,+\, \dim_k \left((f_{\ast} \O_{C'})/\O_C\right)_x
\,+\, \sum_{y \in f^{-1}(x)} \jac(C',y).
\]
In particular, for every closed point $y \in C'$ lying above $x$,
we have $\jac(C',y) \,\leq\, \jac(C,x)$.
\end{prop}

\begin{proof}
We have the following the sequence of homomorphisms of $\O_{C}$-modules
\[
\xymatrix{
\Omega^1_{C/k} \ar[r]^-{g} 
& f_{\ast} \big( \Omega^1_{C'/k}/\Ker(c_{C'/k}) \big)
 \ar[r]^-{f_{\ast} c_{C'/k}}
& f_{\ast} \omega_{C'/k} \ar[r]^-{h} & \omega_{C/k},
}
\]
whose composition is equal to the class map $c_{C/k}$. 
By assumption, $h$ is generically an isomorphism. 
Since $C'$ is reduced and a local complete intersection, 
$\omega_{C'/k}$ is torsion free. 
Hence, $h$ is injective. 
Moreover, $f_{\ast} c_{C'/k}$ is injective.
We obtain the two short exact sequences
\[
\xymatrix{
0 \ar[r]& \Coker(g) \ar[r]
& \Coker(c_{C/k}) \ar[r]
& \Coker(h \circ f_{\ast} c_{C'/k}) \ar[r]
& 0
}
\]
and
\[
\xymatrix{
0 \ar[r]
& f_{\ast} \Coker(c_{C'/k}) \ar[r]
& \Coker(h \circ f_{\ast} c_{C'/k}) \ar[r]
& \Coker(h) \ar[r]
& 0.
}
\]
Taking dimensions of the stalks, we obtain the following equality
\[
\jac(C,x) \,=\, \dim_k \Coker(g)_x
\,+\, \dim_k \Coker(h)_x \,+\, \dim_k \left(f_{\ast} \Coker(c_{C'/k})\right)_x.
\]
By Proposition \ref{Proposition:lci}, the last term is equal to the sum
$\sum_{y \in f^{-1}(x)} \jac(C',y)$.

It remains show
\[ \dim_k \Coker(h)_x
= \dim_k \left((f_{\ast} \O_{C'})/\O_C\right)_x. \]
We put $A := \O_{C,x}$.
Then,
$B := \Gamma(C' \otimes_C \Spec A,\,\O_{C' \otimes_C \Spec A})$
is a finite semi-local $A$-algebra.
We have a short exact sequence of $A$-modules
$0 \to A \to B \to B/A \to 0$.
By assumption, the $A$-module $B/A$ is of finite length.
Since $A$ is a local complete intersection, it is Gorenstein,
and thus, the dualizing module $\omega_A$
is a free $A$-module of rank $1$.
Hence
$\Hom_{A}(B/A, \omega_A) = 0$
and
$H^0_{\idealm_A}(B/A) = B/A$,
where $\idealm_A$ denotes the maximal ideal of $A$.
By Grothendieck's local duality \cite[Theorem 3.5.8]{BrunsHerzog}, we have
$\Ext^1_{A}(B, \omega_A) \,=\, \Hom_A\left(H^0_{\idealm_A}(B), E(A/\idealm_A)\right) \,=\, 0$.
We also have $\Hom_{A}(A, \omega_A) = \omega_A$.
Hence, we obtain the following short exact sequence:
\[
\xymatrix{
0 \ar[r]
& \Hom_{A}(B, \omega_A) \ar[r]^-{h}
& \omega_A \ar[r]
& \Ext^1_{A}(B/A, \omega_A) \ar[r]
& 0.
}
\]
Grothendieck's local duality gives an isomorphism
\[ \Ext^1_{A}(B/A, \omega_A)
   \,\cong\, \Hom_A\left(H^0_{\idealm_A}(B/A), E(A/\idealm_A)\right). \]
The right hand side is the Matlis dual of $B/A$ because
$H^0_{\idealm_A}(B/A) = B/A$.
Since Matlis duality preserves lengths (see \cite[Proposition 3.2.12]{BrunsHerzog}), 
we have
\[ \length_A (\Coker(h)) \,=\, \length_A \Ext^1_{A}(B/A, \omega_A)
   \,=\, \length_A (B/A). \]
Therefore we have
$\dim_k \Coker(h)_x
= \dim_k \left((f_{\ast} \O_{C'})/\O_C\right)_x$.
\end{proof}

\subsection{A criterion for the normalization of a curve being smooth}

In this subsection,
we give a sufficient condition in terms of Jacobian numbers 
for the smoothness of the normalization
of local complete intersection curves.

\begin{thm}
\label{MainTheorem3}
Let $C$ be an integral curve over
a (possibly imperfect) field $k$ of characteristic $p>0$.
Assume that $C$ is local complete intersection over $k$,
and the Jacobian numbers of $C$ 
are strictly less than $p$ at every closed point of $C$.
Let $\pi \colon \widetilde{C} \to C$ be the normalization morphism.
Then $\widetilde{C}$ is smooth over $k$.
\end{thm}

\begin{proof}
Since $\widetilde{C}$ and $\Spec k$ are both regular schemes,
the morphism $\widetilde{C} \to \Spec k$ is a local complete intersection;
see \cite[Chapter 6, Example 3.18]{Liu:Book}.
By Proposition \ref{Proposition:JacobianNumberFiniteMorphism}, we have
$\jac(\widetilde{C},y) \leq \jac(C,x) < p$
for all closed points $x \in C$ and $y \in \widetilde{C}$ with $\pi(y) = x$.
Hence, $\widetilde{C}$ is smooth over $k$ by Proposition \ref{Proposition:SmallJacobianNumberSmoothness}.
\end{proof}

\begin{rem}
Theorem \ref{MainTheorem3} is in some sense optimal;
see Lemma \ref{Lemma:QuasiHyperelliptic} for the construction of
a non-smooth regular curve over an imperfect field
which has a singular point of Jacobian number $p$.
\end{rem}

\subsection{Jacobian numbers and completions}

\begin{prop}
\label{Proposition:JacobianNumberCompletion}
Let $C$ and $C'$ be curves over $k$.
Let $x \in C$ and $x' \in C'$ be closed points, such that the completed local rings 
are isomorphic, i.e., there exists an isomorphism of $k$-algebras
$\widehat{\O}_{C,x} \cong \widehat{\O}_{C',x'}$.
Then, we have
$\jac(C,x) \,=\, \jac(C',x')$.
\end{prop}

\begin{proof}
It is enough to show that the Jacobian number $\jac(C,x)$ can be
calculated in terms of the completed local ring $\widehat{\O}_{C,x}$.
We set $A := \O_{C,x}$ and $\widehat{A} := \widehat{\O}_{C,x}$ and let
$\Omega^1_{C/k,x}$ be the stalk of $\Omega^1_{C/k}$ at $x$.
By definition, we have
\[ \jac(C,x) = \dim_k \left(A/\Fitt^1_A (\Omega^1_{C/k,x})\right). \]
If  $A/\Fitt^1_A (\Omega^1_{C/k,x})$
is a finite dimensional $k$-vector space, then there exists some $N\geq1$ such 
that
\[ \idealm_A^N \subset \Fitt^1_A (\Omega^1_{C/k,x}), \]
where $\idealm_A \subset A$ is the maximal ideal.
Hence, we find
\[
A/\Fitt^1_A (\Omega^1_{C/k,x})
\cong \left(A/\Fitt^1_A (\Omega^1_{C/k,x})\right) \otimes_A \widehat{A}
\cong \widehat{A}/\left(\Fitt^1_A (\Omega^1_{C/k,x}) \otimes_A \widehat{A}\right)
\]
Moreover, we have
\[ \Fitt^1_{A} (\Omega^1_{C/k,x}) \otimes_{A} \widehat{A}
= \Fitt^1_{\widehat{A}} (\Omega^1_{C/k,x} \otimes_A \widehat{A}) \]
since the formation of Fittings ideals commutes with base change.
Combining these results, we find
\[
\jac(C,x) \,=\, \dim_k \left(\widehat{A}/\Fitt^1_{\widehat{A}} (\Omega^1_{C/k,x} \otimes_A \widehat{A})\right).
\]

Note that the right hand side of this equation depends only 
on the completion $\widehat{A} = \widehat{\O}_{C,x}$ since
\[ \Omega^1_{C/k,x} \otimes_A \widehat{A} \cong
\varprojlim_{n} \left(\Omega^1_{\widehat{A}/k} / \idealm_{\widehat{A}}^n \Omega^1_{\widehat{A}/k}\right), \]
where $\idealm_{\widehat{A}} \subset \widehat{A}$ is the maximal ideal;
see \cite[Corollary 12.10]{Kunz}. 
If $\jac(C,x) = \infty$, then we also have 
\[
\jac(C,x) \,=\, \dim_k \left(\widehat{A}/\Fitt^1_{\widehat{A}} (\Omega^1_{C/k,x} \otimes_A \widehat{A})\right),
\]
because the homomorphism
\[
A/\Fitt^1_A (\Omega^1_{C/k,x}) \,\to\, 
\left(A/\Fitt^1_A (\Omega^1_{C/k,x})\right) \otimes_{A} \widehat{A} 
\,\cong\, 
\widehat{A}/\Fitt^1_{\widehat{A}} (\Omega^1_{C/k,x} \otimes_A \widehat{A})
\]
is faithfully flat and hence, injective.
\end{proof}


\begin{cor}
\label{Corollary:JacobianNumberDerivative}
Let $C$ be a curve over $k$ and let $x \in C(k)$ be a $k$-rational point.
If the complete local ring $\widehat{\O}_{C,x}$ is isomorphic to $k[[ S,T ]]/(f)$
for some non-zero formal power series $f \in k[[ S,T ]]$ with $f(0,0) = 0$, 
then we have
$\jac(C,x)  = \dim_k \left(k[[S,T]]/(f_S,f_T,f)\right)$.
Here, $f_S$ and $f_T$ are the derivatives of $f$ with respect to $S$ and $T$, 
respectively.
\end{cor}

\begin{proof}
We set $\widehat{A} := \widehat{\O}_{C,x} \cong k[[S,T]]/(f)$.
By the proof of Proposition \ref{Proposition:JacobianNumberCompletion},
it suffices to calculate the first Fitting ideal of
\[ \varprojlim_{n} \left(\Omega^1_{\widehat{A}/k} / \idealm_{\widehat{A}}^n \Omega^1_{\widehat{A}/k}\right), \]
where $\idealm_{\widehat{A}} \subset \widehat{A}$ is the maximal ideal.
This $\widehat{A}$-module can be calculated as follows: we have
\[
\Omega^1_{\widehat{A}/k} / \idealm_{\widehat{A}}^N \Omega^1_{\widehat{A}/k}
\,\cong\,
\left((\widehat{A}/\idealm_{\widehat{A}}^N) dS \,\oplus\, (\widehat{A}/\idealm_{\widehat{A}}^N) dT\right)/J_N,
\]
where $J_N$ is the $(\widehat{A}/\idealm_{\widehat{A}}^N)$-module generated by
the image of
\[ df := f_S \, dS + f_T \, dT. \]
Taking the projective limit with respect to $N$, we have
\[ \varprojlim_{N} (\Omega^1_{\widehat{A}/k} / \idealm_{\widehat{A}}^N \Omega^1_{\widehat{A}/k})
\,\cong\,
(\widehat{A} \, dS \oplus \widehat{A} \, dT)/\widehat{J}, \]
where $\widehat{J}$ is the $\widehat{A}$-module generated by
the image of $df$ in $\widehat{A} \, dS \oplus \widehat{A} \, dT$.
From this, we see that the first Fitting ideal of the above $\widehat{A}$-module
is generated by $f_S$ and $f_T$.
\end{proof}

\subsection{Jacobian numbers of curves in characteristic $\mathbf{2}$}

As often in characteristic $p$ geometry, the situation is different if $p=2$.

\begin{prop}\label{JacobianChar=2}
Let $C$ be a curve over an algebraically closed field $k$ of characteristic $2$ 
and  let $x \in C$ be a closed point.
If the complete local ring $\widehat{\O}_{C,x}$ is isomorphic to $k[[ S,T ]]/(f)$
for some non-zero formal power series $f \in k[[ S,T ]]$ with $f(0,0) = 0$,
then the Jacobian number $\jac(C,x)$ is different from $2$.
\end{prop}

\begin{proof}
We have
$\jac(C,x) = \dim_k \left(k[[S,T]]/(f_S,f_T,f)\right)$,
where $f_S$ and $f_T$ are the derivatives of $f$ with respect to 
$S$ and $T$, respectively; see Corollary \ref{Corollary:JacobianNumberDerivative}. 
We write $f$ in the form $f = \sum_{i=1}^{\infty} f_i$,
where $f_i$ is homogeneous of degree $i$. 
If $f_1 \neq 0$, then $f_S$ or $f_T$ is a unit in $k[[S, T]]$ and then, we
have $ \jac(C,x) = 0$.
Next, we write $f_2$ as $f_2 = aST+bS^2+cT^2$ for some $a,b,c \in k$. 
If $f_1 = 0$ and $a=0$, then we have
$(f_S,f_T,f) \subset (ST, S^2, T^2)$ (here, we use the condition $p=2$)
and find
$\jac(C,x) \geq \dim_k \left(k[[S,T]]/(ST, S^2, T^2)\right) =3$.
Finally, we assume $f_1 = 0$ and $a \neq 0$.
We claim that $x \in C$ is a node:
we write $f_2$ as
\[
f_2 \,=\, aST+bS^2+cT^2 \,=\, (uS+vT)\cdot (u'S+v'T),
\]
for some $u, v, u', v' \in k$, 
and where  the terms $uS+vT$ and $u'S+v'T$ are linearly independent
since $a \neq 0$ and $p=2$. 
As in \cite[Chapter I, Example 5.6.3]{Hartshorne},
we find an automorphism of $k[[S, T]]$ that sends $S, T$ 
to $g, h$, respectively, such that $f=gh.$
This shows that $x \in C$ is a node and thus, $\jac(C,x) = 1$; 
see Proposition \ref{Proposition:cusp}.
\end{proof}

\subsection{Upper semicontinuity of Jacobian numbers}

\begin{prop}
\label{Proposition:UpperSemicontinuityJacobianNumbers}
Let $U$ be a Noetherian integral scheme.
Let $\pi \colon \mathscr{C} \to U$ be a flat family of
proper and geometrically reduced curves parameterized by $U$.
Let $u_0 \in U$ be a closed point, 
let $N$ be a non-negative integer, 
and assume that the Jacobian numbers of $\mathscr{C}_{u_0}$ 
are smaller than or equal to $N$ at every closed point.
Then, there exists a non-empty open subset
$U'\subset U$  
such that for every point $x\in U'$
(not necessarily closed) the Jacobian numbers of the curve
$\mathscr{C}_x \otimes_{\kappa(x)} \kappa(x)^{\rm{sep}}$
over $\kappa(x)^{\rm{sep}}$ are smaller than or equal to 
$N$ at every closed point.
(We do not require the open subset $U'$ contains the closed point $u_0$.)
\end{prop}

\begin{proof}
Since the smooth locus of $\pi \colon \mathscr{C} \to U$ is dense in every fiber,
the support of $\O_{\mathscr{C}}/\Fitt^1_{\O_{\mathscr{C}}}(\Omega^1_{\mathscr{C}/U})$
has only finitely many closed points in each fiber of $\pi$. 
Let $i_Z \colon Z \hookrightarrow \mathscr{C}$
be the closed subscheme of $\mathscr{C}$ defined by 
$\Fitt^1_{\O_{\mathscr{C}}}(\Omega^1_{\mathscr{C}/U})$.
The morphism $Z \to U$ is finite because it is both proper and quasi-finite.
Let $\eta\in U$ be the generic point, let
$\mathscr{C}_{\eta} := \mathscr{C} \times_U \eta$ be the generic fiber, 
and let $t_1,\ldots,t_n$ be closed points of $\mathscr{C}_\eta$ such that 
$Z_{\eta} \,=\, \{\, t_1,\ldots,t_n \,\}$.

We put $\overline{u}_0 := \Spec \kappa(u_0)^{\rm{sep}} \to U$,
which is a geometric point above $u_0$.
Let $\widetilde{U}_{\overline{u}_0}$ be the strict Henselization of $U$
relative to $\overline{u}_0$.
Then, every connected component of 
$Z \times_{U}{\widetilde{U}_{\overline{u}_0}}$ 
has a unique element above the closed point of 
$\widetilde{U}_{\overline{u}_0}$.
Since the strict Henselization is a direct limit of \'etale neighborhoods of $u_0$,
we may assume, after possibly replacing $U$ by an \'etale neighborhood of $u_0$, 
that every connected component of $Z \times_{U} \Spec \O_{U,u_0}$
has a unique element above $u_0$.
(Here we use Proposition \ref{Proposition:JacobianNumberEtaleCovering}:
for a field extension $k'/\kappa(u_0)$, the Jacobian numbers of
$\mathscr{C}_{u_0} \otimes_{\kappa(u_0)} k'$ are also smaller than or equal to $N$.
Hence, it is enough to prove the assertion after shrinking $U$
and replacing $U$ by an \'etale neighborhood of $u_0$.)

Let $W$ be a connected component of $Z \times_{U} \Spec \O_{U,u_0}$
that intersects non-trivially with the generic fiber $Z_{\eta}$. 
We put $A := \O_{U,u_0}$ and $B := \O_{W}$.
Then, $B$ is a finite $A$-algebra.
Let $s \in W$ be the unique element above $u_0$.
We have
$\jac(\mathscr{C}_{u_0}, s) = \dim_{A/\idealm_A} \left(B \otimes_A (A/\idealm_A)\right)$,
where $\idealm_A \subset A$ is the maximal ideal corresponding 
to $u_0$.
Similarly, for every point $t_i \in W$ in the generic fiber, we have
\[ \jac(\mathscr{C}_{\eta}, t_i) = \dim_{\Frac(A)} \left(B \otimes_A \Frac(A)\right)_{\idealp_i}, \]
where $\idealp_i$
is the prime ideal of $B \otimes_A \Frac(A)$ corresponding to $t_i$.
Then, we have
\begin{align*}
\sum_{t_i \in W} \jac(\mathscr{C}_{\eta}, t_i)
&= \dim_{\Frac(A)} \left(B \otimes_A \Frac(A)\right) \\
&\leq \dim_{A/\idealm_A} \left(B \otimes_A (A/\idealm_A)\right) = \jac(\mathscr{C}_{u_0}, s)
\end{align*}
by Nakayama's lemma.
Since we assumed $\jac(\mathscr{C}_{u_0}, s) \leq N$,
we have $\jac(\mathscr{C}_{\eta}, t_i) \leq N$ for every $t_i \in W$.
This shows the assertion of this proposition for the generic fiber.

Replacing $U$ by an \'etale neighborhood of $\eta$ if necessary,
we may assume that the following three conditions are satisfied:
\begin{enumerate}
\item the residue fields at $t_i$ for $i=1,\ldots,n$ are purely inseparable extensions of $\kappa(\eta)$, 
\item the Zariski closures $\overline{\{ t_i \}} \subset Z$ for $i=1,\ldots,n$,
do not intersect with each other over $U$, and
\item the morphism  $Z \to U$ is flat.
\end{enumerate}
Let $u \in U$ be a point, which is not necessarily closed.
Since the $\overline{\{ t_i \}} \ (1 \leq i \leq n)$ do not intersect 
over $u \in U$ and since the morphism 
$Z \to U$
is flat, 
for each element $s \in Z$ above $u$,
there is a unique integer $i$ with $s \in \overline{\{ t_i \}}$.
We note that the Zariski closure 
$\overline{\{ t_i \}}$ of $t_i$ in $Z \times_{U} \Spec \O_{U,u}$
is a connected component of 
$Z \times_{U} \Spec \O_{U,u}$
and $s$ is the unique element of $\overline{\{ t_i \}}$ above $u$.
Since $\Gamma(\overline{\{ t_i \}}, \O_{Z \times_{U} \Spec \O_{U,u}})$
is a free $\O_{U,u}$-module of finite rank,
by the same argument as before, we have
$\jac(\mathscr{C}_{\eta}, t_i) = \jac(\mathscr{C}_u, s)$.
Hence, we have $\jac(\mathscr{C}_u, s) \leq N$.
\end{proof}

\section{$\delta$-invariants of curves over arbitrary fields}
\label{Section:DeltaInvariant}

In this section, we briefly recall the definition and the basic properties of $\delta$-invariants which we need.
For a curve over an algebraically closed field, we define $\delta$-invariants in the usual way.
For a curve over an imperfect field, we basically only consider the $\delta$-invariants of the base change of the curve to an algebraically closed field because we want to study non-smooth points rather than non-regular points.
Therefore, it is useful to introduce a variant of the $\delta$-invariant, which we call the \textit{geometric} $\delta$-invariant, of a closed point of a curve over an arbitrary field.

Let $C$ be a geometrically reduced curve over a field $k$.
We put $\overline{C}:=C\otimes_k {k^{\alg}}$.
Let $\pi \colon \widetilde{\overline{C}} \to \overline{C}$
be the normalization morphism.
Let $p \colon \overline{C} \to C$ be the natural morphism.

\begin{defn}
\label{Definition:DeltaInvariant}
For a closed point $x\in \overline{C}$, the \textit{$\delta$-invariant} of $\overline{C}$ at $x$ is defined to be
\[ \delta(\overline{C}, x) := \dim_{k^{\alg}
 }(\pi_\ast\O_{\widetilde{\overline{C}}}/\O_{\overline{C}})_x. \]
For a closed point $x \in C$,
the \textit{geometric $\delta$-invariant} of $C$ at $x$
is defined to be
\[ \delta(C,x) \,:=\, \sum_{y \in p^{-1}(x)} \delta(\overline{C}, y). \]
\end{defn}

We collect some basic properties, which can be verified immediately.

\begin{prop}\label{Proposition:DeltaInvariantNormalizationSmooth}
Let $C$ be a geometrically integral curve over a field $k$.
Let $\pi \colon \widetilde{C} \to C$ be the normalization morphism.
If $\widetilde{C}$ is smooth over $k$, we have
$\delta(C,x) = \dim_k(\pi_{\ast}\O_{\widetilde{C}}/\O_C)_x$.
\end{prop}

\begin{prop}
\label{Proposition:DeltaInvariantZeroSmoothness}
Let $C$ be a geometrically reduced curve over a field $k$.
For a closed point $x \in C$,
we have $\delta(C,x) = 0$ if and only if $C$ is smooth at $x$.
\end{prop}

\begin{prop}
\label{Proposition:DeltaInvariantClosedSubscheme}
Let $C$ and $C'$ be geometrically reduced curves over a field $k$ together with a closed immersion 
$i \colon C' \hookrightarrow C$.
For every closed point $x \in C'$, we have the inequality
$\delta(C',x) \leq \delta(C,x)$.
\end{prop}

\begin{prop}
\label{Proposition:DeltaInvariantEtaleCovering}
Let $C$ be a geometrically reduced curve over a field $k$ and let $k'/k$ be a field extension.
We denote by $p \colon C_{k'}=C \otimes_k k' \to C$
the natural morphism.
For every closed point $x \in C$, we have
$\delta(C,x) = \sum_{y \in p^{-1}(x)} \delta(C_{k'}, y)$.
\end{prop}

\begin{prop}
 \label{Proposition:BirationalDeltaInvariant}
 Let $C$ and $C'$ be two geometrically reduced curves over a field $k$.
 Let $f \colon C' \to C$ be a finite morphism over $k$ and assume that there
exists an open dense subset $U \subset C$ such that the restriction
$f|_{f^{-1}(U)} \colon f^{-1}(U) \to U$
is an isomorphism.
 If $x \in C$ is a closed point, then
$\delta(C, x) = \dim_k (f_\ast\O_{C'}/\O_C)_x + \sum_{y\in f^{-1}(x)} \delta(C', y)$.
\end{prop}

We give a sufficient criterion in terms of
$\delta$-invariants for the smoothness of the normalization of a curve over a possibly imperfect field.

\begin{thm}
\label{Theorem:TateGenusChange}
 Let $C$ be a regular and geometrically integral curve over
 a (possibly imperfect) field $k$ of characteristic $p > 0$.
 We put $\overline{C}:=C \otimes_k k^{\alg}$.
 Assume that $\delta(\overline{C}, x)<(p-1)/2$ for every closed point $x \in \overline{C}$.
 Then $C$ is smooth over $k$.
\end{thm}

\begin{proof}
By Proposition \ref{Proposition:DeltaInvariantEtaleCovering}, after replacing $k$ by its finite separable extension, we may assume that $\delta(C, x)<(p-1)/2$ for every closed point $x \in C$.
We choose a finite extension $k'/k$, such that the normalization $\widetilde{C_{k'}}$ of $C_{k'}$ is smooth over $k'$.
We have to show that $\O_{C_{k'}, x}$ is regular for every closed point $x \in C_{k'}$.
We fix a closed point $x \in C_{k'}$ and set $A:=\O_{C_{k'}, x}$.
Let $B$ be the normalization of $A$, which is a finite semi-local $A$-module.
We will use the same notation as in the proof of Proposition \ref{Proposition:JacobianNumberFiniteMorphism}.
The conductor ideal $I \subset A$ is defined by the image of the map
$h \colon \Hom_A(B, A) \to A$, $\phi \mapsto \phi(1)$.
It turns out that $I$ is an ideal of $B$.
As in the proof of Proposition \ref{Proposition:JacobianNumberFiniteMorphism}, we have
\[ \length_A (\Coker(h)) = \length_A (B/A). \]
Since $A/I$ is isomorphic to $\Coker(h)$ as an $A$-module, we have
\[ \length_A (A/I) = \length_A (B/A). \]
By the following short exact sequence of $A$-modules
\[ 0 \to A/I \to B/I \to B/A \to 0, \]
we have
\[ \length_A (B/I) = 2\cdot\length_A (B/A). \]
If $A$ is not regular, we have $\length_A (B/I) \geq p-1$ by \cite[Theorem 1.2]{PatakfalviWaldron}.
This implies
\[ \dim_{k'}(B/A)=[\kappa(x):k']\cdot\length_A (B/A) \geq (p-1)/2. \]
Since $\widetilde{C_{k'}}$ is smooth,
we have
\[ \delta(C_{k'}, x)=\dim_{k'}(B/A) \]
by Proposition \ref{Proposition:DeltaInvariantNormalizationSmooth}.
Thus, we find
$\delta(C_{k'}, x) \geq (p-1)/2$.
This contradicts the assumption by Proposition \ref{Proposition:DeltaInvariantEtaleCovering}.
Hence $A$ is regular.
\end{proof}

\begin{rem}
Theorem \ref{Theorem:TateGenusChange} is in some sense optimal;
see Lemma \ref{Lemma:QuasiHyperelliptic} for the construction of
a non-smooth regular curve over an imperfect field
which has a singular point of geometric $\delta$-invariant $(p-1)/2$.
\end{rem}

\begin{rem}
Theorem \ref{Theorem:TateGenusChange}
is a classical result of Tate \cite{Tate} if the sum over all $\delta$-invariants of $\overline{C}$ is strictly less than $(p-1)/2$; see also \cite{Schroeer}.
Thus, our result is a slight improvement over Tate's result.
In \cite{Schroeer:GenusChange}, Schr\"oer gave a simple proof of Tate's theorem.
Our proof, which is in terms of the \textit{local} $\delta$-invariants $\delta(\overline{C},x)$,
relies ideas from work of Patakfalvi and Waldron \cite{PatakfalviWaldron}.
\end{rem}

The upper semicontinuity of geometric $\delta$-invariants is presumably well-known to the experts.
The following is all we need.

\begin{prop}
\label{Proposition:UpperSemicontinuityDeltaInvariant}
Let $U$ be a Noetherian integral scheme, and let $\eta \in U$ be the generic point.
Let $\pi \colon \mathscr{C} \to U$ be a flat family of
proper and geometrically reduced curves parameterized by $U$ such that the generic fiber $\mathscr{C}_\eta$ is geometrically irreducible over $\kappa(\eta)$.
Let $u_0 \in U$ be a closed point,
let $N$ be a non-negative integer, 
and assume that the geometric $\delta$-invariants of $\mathscr{C}_{u_0}$ 
are smaller than or equal to $N$ at every closed point.
Then, there exists a non-empty open subset
$U'\subset U$  
such that for every point $x\in U'$
(not necessarily closed) the geometric $\delta$-invariants of the curve
$\mathscr{C}_x \otimes_{\kappa(x)} \kappa(x)^{\rm{sep}}$
over $\kappa(x)^{\rm{sep}}$ are smaller than or equal to 
$N$ at every closed point.
\end{prop}

\begin{proof}
First, we show that the geometric $\delta$-invariants of the curve
$\mathscr{C}_\eta \otimes_{\kappa(\eta)} \kappa(\eta)^{\rm{sep}}$
over $\kappa(\eta)^{\rm{sep}}$ are smaller than or equal to 
$N$ at every closed point.
By Proposition \ref{Proposition:DeltaInvariantEtaleCovering}, we may assume that $U$ is the spectrum of a complete discrete valuation ring $A$, whose residue field corresponds to the closed point $u_0$ of $U$.
Moreover, after replacing $A$ by a finite extension, we may assume that the normalization of $\mathscr{C}_\eta$ (resp.\ $\mathscr{C}_{u_0}$) is smooth over $\kappa(\eta)$ (resp.\ $\kappa(u_0)$); see \cite[Proposition 17.15.14]{EGA4-4}.
Let $\pi \colon \widetilde{\mathscr{C}} \to \mathscr{C}$ be the normalization morphism.
As in the proof of Lemma \ref{Lemma:StableMapFamilyImage}, we have the following equality of $1$-cycles on $\mathscr{C}$:
\[ \pi_{\ast}[(\widetilde{\mathscr{C}})_{u_0}]=[\mathscr{C}_{u_0}]; \]
see \cite[Theorem 7.2.18]{Liu:Book}.
From this equality, we see that $(\widetilde{\mathscr{C}})_{u_0}$ is generically reduced.
Since $(\widetilde{\mathscr{C}})_{u_0}$ has no embedded points by \cite[Proposition 7.2.15 and Corollary 7.2.22]{Liu:Book}, it follows that $(\widetilde{\mathscr{C}})_{u_0}$ is reduced.
Let now $\widetilde{\mathscr{C}_{u_0}}$ be the normalization of $\mathscr{C}_{u_0}$.
By the above equality again, the normalization morphism factors as
\[ \widetilde{\mathscr{C}_{u_0}} \to (\widetilde{\mathscr{C}})_{u_0} \to  \mathscr{C}_{u_0}. \]
Thus, we have 
\[ \dim_{\kappa(u_0)}(\mathscr{F}\vert_{\mathscr{C}_{u_0}})_x \leq \delta(\mathscr{C}_{u_0}, x) \leq N \]
for every closed point $x \in \mathscr{C}_{u_0}$.
By considering a closed subscheme $i_{Z} \colon Z \hookrightarrow \mathscr{C}$ such that $\mathscr{F}$ comes from a coherent sheaf $\mathscr{F}_Z$ on $Z$ and $Z=\mathrm{Supp}(\mathscr{F})$, the similar arguments as in the proof of Proposition \ref{Proposition:UpperSemicontinuityJacobianNumbers} show that the claim is true.

Next, we show that the just established result implies the existence of an open subset $U' \subset U$ as in the assertion.
There is a flat morphism of finite type $f \colon U'' \to U$, such that the normalization of $\mathscr{C} \times_{U} U''$ is smooth over $U''$.
Since $f$ is an open map, we may assume that the normalization $\widetilde{\mathscr{C}}$ of $\mathscr{C}$ is smooth over $U$.
Let $\pi \colon \widetilde{\mathscr{C}} \to \mathscr{C}$ be the normalization morphism.
Now, by considering $\mathscr{F}:=\pi_{\ast}\O_{\widetilde{\mathscr{C}}}/\O_{\mathscr{C}}$ and a closed subscheme $i_{Z} \colon Z \hookrightarrow \mathscr{C}$ as above, similar arguments as in the proof of Proposition \ref{Proposition:UpperSemicontinuityJacobianNumbers} show the existence of an open subset $U' \subset U$ as desired.
\end{proof}

\section{The key lemma}
\label{Section:KeyLemma}
In this section, we prove a lemma,
which is used in the proofs of Theorem \ref{MainTheorem2} 
and Theorem \ref{MainTheorem1}.
This lemma is the technical heart of this article.

\begin{lem}
\label{KeyLemma}
Let $k$ be an algebraically closed field $k$ of characteristic $p>0$.
Let $X$ be a smooth, proper, and connected variety $X$ over $k$ with $\dim(X) \geq 2$
that is dominated by a map from a family of rational curves,
i.e., there exists a pair $(\pi,\varphi)$ as in Definition \ref{Definition:FamilyRationalCurve} (iv)
such that $\dim(U)=\dim(X)-1$ and $\varphi \colon \mathscr{C} \to X$ is dominant.
Assume moreover that $\mathscr{C}$ and $U$ are normal.
Then, after possibly shrinking $U$,
there exists a commutative diagram
\[
\xymatrix{
\mathscr{C}  \ar^-{}[r] \ar@/^16pt/[rr]^-{\varphi} \ar[d]_-{\pi} &
\mathscr{C}' \ar[d]_-{s} \ar[r]^-{} & X \\
U \ar[r]^-{t} & U'
}
\]
satisfying the following conditions:
\begin{enumerate}
\item $\mathscr{C}'$ and $U'$ are normal and connected varieties over $k$.
\item $s \colon \mathscr{C}' \to U'$ is a proper flat morphism, and $t \colon U \to U'$ is a finite morphism.
\item $\mathscr{C} \to \mathscr{C}'$ is a finite morphism.
\item $k(\mathscr{C}')$ is the separable closure of $k(X)$ in $k(\mathscr{C})$.
\item $k(U')$ is algebraically closed in $k(\mathscr{C}')$.
\item For every closed point $u' \in U'$,
$s^{-1}(u')_{\mathrm{red}}$ is a (possibly singular) rational curve.
\end{enumerate}
\end{lem}

\begin{lem}
\label{Lemma:Radicial}
Let $X$ and $Y$ be integral schemes of characteristic $p>0$, and let $f \colon X \to Y$ be a finite and dominant morphism.
Assume that $Y$ is normal and that $f$ is purely inseparable, i.e.,
the finite extension $k(X)/k(Y)$ of function fields induced by $f$
is purely inseparable. 
Then $f$ is radicial.
\end{lem}

\begin{proof}
We may assume that $X$ and $Y$ are affine, say, $X := \Spec A$ and $Y := \Spec B$.
Since the extension $k(X)/k(Y)$ is finite and purely inseparable, 
there exists a positive integer $e \geq 1$, such that $k(X)^{p^e} \subset k(Y).$
Since $A$ is integral over $B$ and $B$ is normal, we have
$A^{p^e} \subset B$.
Let $\mathfrak{q}$ be a prime ideal of $B$.
It follows that $\mathfrak{p}:=\sqrt{\mathfrak{q}A}$ is the unique prime 
ideal of $A$ above $\mathfrak{q}$. 
Hence, $\Spec A \to \Spec B$ is bijective.
Let $\kappa(\mathfrak{p})$ and $\kappa(\mathfrak{q})$ be the residue fields of $\mathfrak{p}$ and
$\mathfrak{q}$, respectively.
Since $A^{p^e} \subset B$, we have $\kappa(\mathfrak{p})^{p^e} \subset \kappa(\mathfrak{q})$,
and the extension $\kappa(\mathfrak{p})/\kappa(\mathfrak{q})$ is purely inseparable. 
This concludes that $f$ is radicial.
\end{proof}

Now, we shall prove Lemma \ref{KeyLemma}.
With the assumptions and notations as in Lemma \ref{KeyLemma},
we can compactify $U$ and $\mathscr{C}$ compatibly by the following claim.

\begin{claim}
\label{Claim:Family1}
There exists a commutative diagram
\[
\xymatrix{
\mathscr{C} \ar@{^{(}->}[r] \ar[d]_-{\pi} \ar@/^16pt/[rr]^-{\varphi} &
\overline{\mathscr{C}} \ar[d]_-{\overline{\pi}} \ar[r]_-{\overline{\varphi}} & X \\
U \ar@{^{(}->}[r] & \overline{U}
}
\]
satisfying the following conditions:
\begin{enumerate}
\item $\overline{\mathscr{C}}$ is a normal, proper, and connected variety over $k$
and $\mathscr{C} \subset \overline{\mathscr{C}}$ is an open subset.
\item $\overline{U}$ is a normal, proper, and connected variety over $k$
and $U \subset \overline{U}$ is an open subset.
\end{enumerate}
\end{claim}

\begin{proof}
Choose a compactification $\overline{\mathscr{C}} \supset \mathscr{C}$. 
Replacing $\overline{\mathscr{C}}$ by
the Zariski closure of the image of
$\mathscr{C} \to  \overline{\mathscr{C}}\times X$, 
we may assume that $\varphi$ extends to a morphism $\overline{\varphi} \colon \overline{\mathscr{C}} \to X$. 
Take a normal compactification $\overline{U} \supset U$.
Replacing $\overline{\mathscr{C}}$ by
the normalization of the Zariski closure of the image of
$\mathscr{C} \to \overline{U} \times \overline{\mathscr{C}}$,
we may assume that $\overline{\mathscr{C}}$ is normal and that
$\pi$ extends to a morphism $\overline{\pi} \colon \overline{\mathscr{C}} \to \overline{U}$.
\end{proof}

The next step is to shrink $U$ and to replace $\overline{\mathscr{C}}$ further in order 
to find a nice factorization
$\overline{\mathscr{C}} \to \overline{\mathscr{C}}' \to X$
of $\overline{\varphi}$.
The rough idea is to take a proper and normal model
of the separable closure of the function field $k(X)$ in $k(\overline{\mathscr{C}})$.
But the actual argument given below is more involved
because we also want to ensure that
the intermediate variety $\overline{\mathscr{C}}'$ admits
an open and dense subset $\mathscr{C}' \subset \overline{\mathscr{C}}'$
that is equipped with a fibration
$s \colon \mathscr{C}' \to U'$
over a normal and connected variety $U'$ such that for every closed point $u' \in U'$,
the reduced closed subscheme $s^{-1}(u')_{\mathrm{red}}$ of
the fiber $s^{-1}(u')$ is a (possibly singular) rational curve.
The delicate point is that such a fibration might not exist if we
start from an arbitrary normal and proper model.

Let $\overline{\mathscr{C}}'$ be the normalization of $X$ in
the separable closure of $k(X)$ in $k(\overline{\mathscr{C}})$.
We denote by $\overline{\varphi}' \colon \overline{\mathscr{C}} \to \overline{\mathscr{C}}'$ and
$\overline{\varphi}'' \colon \overline{\mathscr{C}}' \to X$ the induced morphisms.
Then, we have
\[ \overline{\varphi} = \overline{\varphi}'' \circ \overline{\varphi}'. \]
The morphism $\overline{\varphi}''$ is finite and $k(\overline{\mathscr{C}}')/k(X)$ is separable.
On the other hand, $\overline{\varphi}'$ is a generically finite and proper morphism
and $k(\overline{\mathscr{C}})/k(\overline{\mathscr{C}}')$ is purely inseparable.

The morphism $\overline{\varphi}'$ might not be flat.
We now modify $\overline{\mathscr{C}}$ and $\overline{\mathscr{C}}'$
to obtain a flat morphism as follows:
we apply the flattening theorem of Raynaud-Gruson
\cite[Th\'eor\`eme 5.2.2]{RaynaudGruson}.
(See also \cite[Section 2.19]{deJong:Alteration}.)
Then, we obtain a proper and birational morphism
$\overline{g}' \colon \overline{\mathscr{C}}'_2 \to \overline{\mathscr{C}}'$
such that the strict transform
$\overline{\mathscr{C}}_2 \subset \overline{\mathscr{C}} \times_{\overline{\mathscr{C}}'} \overline{\mathscr{C}}'_2$
is \textit{flat} over $\overline{\mathscr{C}}'_2$.
We denote by
$\overline{\varphi}'_2 \colon \overline{\mathscr{C}}_2 \to \overline{\mathscr{C}}'_2$ and
$\overline{g} \colon \overline{\mathscr{C}}_2 \to \overline{\mathscr{C}}$ the induced morphisms.
Then, the following diagram commutes:
\[
\xymatrix{
\overline{\mathscr{C}}_2 \ar^-{\overline{\varphi}'_2}[r] \ar[d]_-{\overline{g}} & \overline{\mathscr{C}}'_2 \ar[d]_-{\overline{g}'} \\
\overline{\mathscr{C}} \ar[r]^-{\overline{\varphi}'} & \overline{\mathscr{C}}' \ar[r]^-{\overline{\varphi}''} & X.
}
\]
Here, $\overline{\varphi}'_2 \colon \overline{\mathscr{C}}_2 \to \overline{\mathscr{C}}'_2$
is \textit{finite} because it is proper, flat, and generically finite.

The varieties $\overline{\mathscr{C}}_2$ and $\overline{\mathscr{C}}'_2$ might not be normal.
Passing to normalizations, we find normal and proper connected varieties
$\overline{\mathscr{C}}_3$ and $\overline{\mathscr{C}}'_3$
and a morphism
$\overline{\psi} \colon \overline{\mathscr{C}}_3 \to \overline{\mathscr{C}}'_3$
over $k$ and we obtain the following commutative diagram:
\[
\xymatrix{
\overline{\mathscr{C}}_3 \ar^-{\overline{\psi}}[r] \ar[d]_-{\overline{h}} & \overline{\mathscr{C}}'_3 \ar[d]_-{\overline{h}'} \\
\overline{\mathscr{C}}_2 \ar^-{\overline{\varphi}'_2}[r] \ar[d]_-{\overline{g}} & \overline{\mathscr{C}}'_2 \ar[d]_-{\overline{g}'} \\
\overline{\mathscr{C}} \ar[r]^-{\overline{\varphi}'} & \overline{\mathscr{C}}' \ar[r]^-{\overline{\varphi}''} & X.
}
\]
Here, $\overline{h}$ and $\overline{h}'$ are proper birational morphisms.
Since $\overline{\varphi}'_2 \circ \overline{h}$ is finite and $\overline{h}'$ is separated, the morphism $\overline{\psi}$ is finite.
Let us summarize the situation:
\begin{enumerate}
\item $\overline{\mathscr{C}},\overline{\mathscr{C}}',\overline{\mathscr{C}}_3,\overline{\mathscr{C}}'_3$ are normal, proper, and connected varieties over $k$.
\item $\overline{g}, \overline{g}', \overline{h}, \overline{h}'$ are proper birational morphisms.
\item $\overline{\varphi}'_2$ and $\overline{\psi}$ are finite morphisms.
\end{enumerate}

Since $\overline{g} \circ \overline{h}$
is an isomorphism outside a closed subset of codimension $\geq 2$ in $\overline{\mathscr{C}}$,
after removing its image in $\overline{U}$ from $U$,
we may assume that the restriction 
$\mathscr{C}_3:=(\overline{g} \circ \overline{h})^{-1}(\mathscr{C})
\to \mathscr{C}$
of $\overline{g} \circ \overline{h}$ is an isomorphism.
By Lemma \ref{Lemma:Radicial},
the morphism $\overline{\psi}$ is radicial.
Hence it is a homeomorphism.
The image $V' := \overline{\psi}(\mathscr{C}_3) \subset \overline{\mathscr{C}}'_3$
is open, and
the induced morphism
$
\overline{\psi} \colon \mathscr{C}_3 \to V'
$
is finite.
We obtain the following diagram:
\[
\xymatrix{
\mathscr{C} \cong \mathscr{C}_3
\ar^-{\overline{\psi}}[r] \ar[d]_-{\pi} & 
V'
\ar[rr]^-{\overline{\varphi}'' \circ \overline{g}' \circ \overline{h}'} && X \\
U
}
\]
Shrinking $U$ further, we may assume that $U$ is affine.
We set
$U' := \Spec H^0(V', \O_{V'})$.
By a lemma of Tanaka \cite[Lemma A.1]{Tanaka},
$U'$ is a normal connected variety over $k$
equipped with a proper surjective morphism
$s \colon V' \to U'$
and a finite surjective morphism $t \colon U \to U'$
such that the following diagram commutes:
\[
\xymatrix{
\mathscr{C} \cong \mathscr{C}_3
\ar^-{\overline{\psi}}[r] \ar[d]_-{\pi} & 
V'
\ar[rr]^-{\overline{\varphi}'' \circ \overline{g}' \circ \overline{h}'}
\ar[d]_-{s}
&& X \\
U \ar[r]^-{t} & U'
}
\]
By construction, we have 
$s_{\ast}\O_{V'} \cong \O_{U'}$.
It follows that
$k(U')$ is algebraically closed in $k(V')$.
For every closed point $u' \in U'$,
the scheme $s^{-1}(u')_{\mathrm{red}}$ is
a (possibly singular) rational curve because
it is dominated by a geometric fiber of $\pi$.
After possibly shrinking $U'$ further, we may assume $s$ is flat.
Putting $\mathscr{C}' := V'$, all the assertions are proved,
which establishes Lemma \ref{KeyLemma}.

\section{Proof of the main theorems}
\label{Section:ProofMainTheorem}

In this section, we will prove Theorem \ref{MainTheorem1} and Theorem \ref{MainTheorem2}.
First, we will prove Theorem \ref{MainTheorem2}, which actually follows from the following,
more general result for maps from a family of rational curves.

\begin{thm}
\label{MainTheoremMaps:HigherDimension}
Let $k$ be an algebraically closed field $k$ of characteristic $p>0$.
Let $X$ be a smooth, proper, and connected variety $X$ over $k$ with $\dim(X) \geq 2$.
Assume that there exists a pair $(\pi,\varphi)$
as in Definition \ref{Definition:FamilyRationalCurve} (iv)
satisfying the following conditions: 
\begin{enumerate}
\item $\dim(U)=\dim(X)-1$,
\item $\varphi \colon \mathscr{C} \to X$ is dominant, and
\item $k(\mathscr{C}) \cap k(X)^{\sep}$ is a separable extension of $k(U) \cap k(X)^{\sep}$.
\end{enumerate}
Moreover, we assume the pair $(\pi,\varphi)$
satisfies \textit{at least one} of the following conditions:
\begin{enumerate}
\setcounter{enumi}{3}
\item For every closed point $u \in U$, the $\delta$-invariants of $\varphi(\mathscr{C}_{u})$ are strictly less than $(p-1)/2$ at every closed point.
\item For every closed point $u \in U$,
$\varphi(\mathscr{C}_{u})$ is a local complete intersection rational curve on $X$, all of whose Jacobian numbers are strictly less than $p$ at every closed point.
\end{enumerate}
Then, $X$ is separably uniruled and thus, $X$ has negative Kodaira dimension.
\end{thm}

\begin{proof}[\bf \em Proof of Theorem \ref{MainTheoremMaps:HigherDimension}]
After possibly shrinking $U$, we may assume by Lemma \ref{Lemma:StableMapFamilyImage} 
that  the image $W$ of $\mathscr{C} \to U \times X$ is a flat family of rational curves.
After replacing $\mathscr{C}$ by the normalization of $W$ and possibly shrinking $U$ further, 
we may assume that $\mathscr{C}$ and $U$ are normal varieties
and that the morphism $\varphi_u \colon \mathscr{C}_u \to X$
is a generic immersion for every closed point $u \in U$.
After possibly shrinking $U$ even further,
there exists a commutative diagram
\[
\xymatrix{
\mathscr{C}  \ar^-{}[r] \ar@/^16pt/[rr]^-{\varphi} \ar[d]_-{\pi} &
\mathscr{C}' \ar[d]_-{s} \ar[r]^-{} & X \\
U \ar[r]^-{t} & U'
}
\]
as in Lemma \ref{KeyLemma}.
We have
\[ k(U) \cap k(X)^{\sep}=k(U) \cap k(\mathscr{C}')=k(U') \]
since $k(U')$ is algebraically closed in $k(\mathscr{C}')$.
Hence the extension $k(\mathscr{C}')/k(U')$ is separable by our assumptions.
After shrinking $U'$ again, we may assume that the varieties $\mathscr{C}'$ and $U'$ 
are smooth over $k$.
By \cite[Theorem 7.1]{Badescu},
after shrinking $U'$, the fibers $s^{-1}(u')$ are geometrically reduced for every closed point $u' \in U'$; see also \cite[Th\'eor\`eme 12.2.4]{EGA4-3}.
Shrinking $U'$ again, we may assume that $s \colon \mathscr{C}' \to U'$ is flat.

Now, we assume that the assumption (v) of Theorem \ref{MainTheoremMaps:HigherDimension} holds.
By Proposition \ref{Proposition:JacobianNumberFiniteMorphism},
the Jacobian numbers of $\mathscr{C}'_{u'}$
are also strictly less than $p$ at every closed point.
(The fiber $\mathscr{C}'_{u'}$ is a local complete intersection because it is 
the fiber of $s \colon \mathscr{C}' \to U'$ and both, $\mathscr{C}'$ and $U'$, 
are smooth varieties.)
Let $K' := k(U')$ be the function field of $U'$ and let 
$K'^{\rm{sep}}$ be a separable closure of $K'$.
Let $\mathscr{C}'_{K'}$ be the generic fiber of $s \colon \mathscr{C}' \to U'$.
By Proposition \ref{Proposition:UpperSemicontinuityJacobianNumbers},
the Jacobian numbers of
$\mathscr{C}'_{K'^{\rm{sep}}} \,:=\, (\mathscr{C}'_{K'}) \otimes_{K'} K'^{\rm{sep}}$
are strictly less than $p$ at every closed point.
Since $\mathscr{C}'$ is a smooth variety, it is regular.
Hence, the generic fiber $\mathscr{C}'_{K'}$ is regular.
By \cite[Proposition 6.7.4 (a)]{EGA4-2}, $\mathscr{C}'_{K'^{\rm{sep}}}$ is also regular.
By Proposition \ref{Proposition:SmallJacobianNumberSmoothness},
$\mathscr{C}'_{K'^{\rm{sep}}}$ is smooth over $K'^{\rm{sep}}$ and
therefore, $\mathscr{C}'_{K'}$ is smooth over $K'$.
After replacing $U'$ by an \'etale neighborhood, we may assume that
$\mathscr{C}'_{K'}$ has a $K'$-rational point.
Then, $\mathscr{C}'_{K'}$ is isomorphic to the projective line $\P^1_{K'}$ over $K'$,
see \cite[Lemma 11.8]{Badescu}.
This implies that $\mathscr{C}'$ is birationally equivalent to 
$\P^1 \times U'$ over $k$.
Since $k(\mathscr{C}')/k(X)$ is separable, we conclude that $X$ is separably uniruled,
as desired.

When the assumption (iv) of Theorem \ref{MainTheoremMaps:HigherDimension} holds, we can argue similarly by using Theorem \ref{Theorem:TateGenusChange},
Proposition \ref{Proposition:BirationalDeltaInvariant}, and Proposition \ref{Proposition:UpperSemicontinuityDeltaInvariant} instead of Proposition \ref{Proposition:SmallJacobianNumberSmoothness}, Proposition \ref{Proposition:JacobianNumberFiniteMorphism}, and Proposition \ref{Proposition:UpperSemicontinuityJacobianNumbers}, respectively.
\end{proof}

Next, we show how Theorem \ref{MainTheoremMaps:HigherDimension} implies 
Theorem \ref{MainTheorem2}.

\begin{proof}[\bf \em Proof of Theorem \ref{MainTheorem2}]
Assume that $\mathscr{C} \subset U \times X$
gives a family of rational curves on $X$
satisfying the assumptions of Theorem \ref{MainTheorem2}. 
We will show that $\mathscr{C}$ satisfies the assumptions of 
Theorem \ref{MainTheoremMaps:HigherDimension} after possibly shrinking $U$.
When the assumption (v) of Theorem \ref{MainTheorem2} holds, we have to show that, after shrinking $U$, 
for \textit{every} closed point $u \in U$,
$\mathscr{C}_{u}$ is a local complete intersection rational curve on $X$,
all of  whose Jacobian numbers at closed points are strictly less than $p$.

By the assumptions of Theorem \ref{MainTheorem2},
there exists a closed point $u_0 \in U$ such that
$\mathscr{C}_{u_0}$ is a local complete intersection rational curve on $X$,
all of whose Jacobian numbers at closed points are strictly less than $p$. 
Since the fibers of $\mathscr{C} \to U$ is reduced,
after shrinking $U$, 
we may assume that, for \textit{every} closed point $u \in U$,
the Jacobian numbers of $\mathscr{C}_{u}$ are strictly less than $p$ at every closed point,
by Proposition \ref{Proposition:UpperSemicontinuityJacobianNumbers}.
Since $\mathscr{C}_{u_0}$ is a local complete intersection,
the generic fiber $\mathscr{C}_{\eta}$ is a local complete intersection over $\kappa(\eta)$
and hence, after shrinking $U$, we may assume that for \textit{every} closed point $u \in U$,
$\mathscr{C}_{u}$ is a local complete intersection rational curve.

When the assumption (iv) of Theorem \ref{MainTheorem2} holds, by the same arguments as before, we can show that, after shrinking $U$, 
for \textit{every} closed point $u \in U$, the $\delta$-invariants of $\mathscr{C}_{u}$ are strictly less than $(p-1)/2$ at closed points by using Proposition \ref{Proposition:UpperSemicontinuityDeltaInvariant}.
\end{proof}

Finally, we prove Theorem \ref{MainTheorem1}.
This follows from the following more general result for 
maps from a family of curves with rational components by Proposition \ref{Proposition:UpperSemicontinuityJacobianNumbers} and Proposition \ref{Proposition:UpperSemicontinuityDeltaInvariant}.

\begin{thm}
\label{MainTheoremMaps:Surface}
Let $k$ be an algebraically closed field $k$ of characteristic $p>0$.
Let $X$ be a smooth, proper, and connected surface $X$ over $k$.
Let $(\pi, \varphi)$ be 
a map from a family of curves with rational components
as in Definition \ref{Definition:FamilyRationalCurve} (vi)
such that $\varphi \colon \mathscr{C} \to X$ is dominant.
Moreover, we assume that \textit{at least one} of the following conditions is satisfied:
\begin{enumerate}
\item $\dim(U)=1$, there exists a closed point $u_0 \in U$ such that
$\varphi|_{\mathscr{C}_{u_0}} \colon \mathscr{C}_{u_0} \to X$
is a generic immersion,
and the $\delta$-invariants of $\varphi(\mathscr{C}_{u_0})$
are strictly less than $(p-1)/2$ at every closed point.
\item For every closed point $u \in U$,
the $\delta$-invariants of $\varphi(\mathscr{C}_{u})$ are strictly less than $(p-1)/2$ at every closed point.
\item $\dim(U)=1$, there exists a closed point $u_0 \in U$ such that
$\varphi|_{\mathscr{C}_{u_0}} \colon \mathscr{C}_{u_0} \to X$
is a generic immersion,
and the Jacobian numbers of $\varphi(\mathscr{C}_{u_0})$
are strictly less than $p$ at every closed point.
\item For every closed point $u \in U$, the Jacobian numbers of $\varphi(\mathscr{C}_{u})$ are strictly less than $p$ at every closed point.
\end{enumerate}
Then $X$ is separably uniruled and thus, $X$ has negative Kodaira dimension.
\end{thm}

\begin{proof}
We only show Theorem \ref{MainTheoremMaps:Surface} under the assumption of condition (iii) or (iv) because the proofs of the other cases are similar.

We first show that Theorem \ref{MainTheoremMaps:Surface} under the assumption of condition (iv) implies 
Theorem \ref{MainTheoremMaps:Surface} under the assumption of condition (iii).
To see this, assume that $X,U,\mathscr{C}$ satisfy the condition (iii) of Theorem \ref{MainTheoremMaps:Surface}.
Let $\eta \in U$ be the generic point.
Replacing $U$ by a finite covering of it, 
we may assume that $U$ is a smooth and connected curve and irreducible components of $\mathscr{C}_{\eta}$ are geometrically irreducible over $\kappa(\eta)$.
Then, we replace $\mathscr{C}$ by an irreducible component that dominates $X$,
and we let
$W := (\pi \times \varphi)(\mathscr{C})$
be the image of $(\pi \times \varphi)|_{\mathscr{C}}$
endowed with the reduced induced subscheme structure.
By Lemma \ref{Lemma:StableMapFamilyImage} (ii),
after replacing $U$ by an open subset of $U$ containing $u_0$, we may assume 
that the fiber $\pr_1^{-1}(u)$ is reduced for every $u \in U$.
By Proposition \ref{Proposition:JacobianNumberClosedSubscheme} and
Proposition \ref{Proposition:UpperSemicontinuityJacobianNumbers},
after shrinking $U$ further if necessary,
we may assume that for every closed point $u \in U$
the Jacobian numbers of $\pr_1^{-1}(u) = \varphi(\mathscr{C}_{u})$
are strictly less than $p$ at every closed point.
Therefore, $X,U,\mathscr{C}$ satisfy the condition (iv)
of Theorem \ref{MainTheoremMaps:Surface}.

It remains to establish Theorem \ref{MainTheoremMaps:Surface} under the assumptions of condition (iv).
Assume that $X,U,\mathscr{C}$ satisfy the conditions of Theorem \ref{MainTheoremMaps:Surface} (iv).
There is a curve $C \subset U$ such that the family $\mathscr{C} \times_{C}U$ dominates $X$.
Hence we may assume that $U$ is a smooth and connected curve over $k$.
By Proposition \ref{Proposition:JacobianNumberClosedSubscheme},
after replacing $U$ by a finite covering of it, shrinking $U$, and replacing $\mathscr{C}$ by an irreducible component that dominates $X$, 
we may assume that $\mathscr{C} \to U$ is a flat family of rational curves. 
Since $X$ is a smooth surface, 
$\varphi(\mathscr{C})_u$ is a local complete intersection
for every closed point $u \in U$.
Since $U$ is a curve,
$k(\mathscr{C}) \cap k(X)^{\sep}$ is a separable extension of $k(U) \cap k(X)^{\sep}$;
see \cite[Theorem 2]{MacLane} (see also \cite[Lemma 7.2]{Badescu}).
Hence, $X$ satisfies the assumptions of Theorem \ref{MainTheoremMaps:HigherDimension}
which implies that $X$ is separably uniruled.
\end{proof}

\section{Examples}
\label{Section:Examples}

In this section, we give some examples illustrating
Theorem \ref{MainTheorem1} and Theorem \ref{MainTheorem2}.
In particular, we show that these results are in some sense optimal and that naive
generalizations are false.
We work over an algebraically closed field $k$
of characteristic $p \geq 0$.

\subsection{An easy corollary}


\begin{cor}\label{cor:easy}
Assume $p >0$. Let $X$ be a smooth, proper, and connected surface of non-negative Kodaira dimension over $k$.
Let $C\subset X$ be a rational curve with
\[ C^2+K_X\cdot C<p-3. \]
Then, $C$ is topologically rigid.
\end{cor}

\begin{proof}
By the adjunction formula \cite[Theorem 9.1.37]{Liu:Book}, the arithmetic genus of $C$ satisfies $p_a(C)<(p-1)/2$.
This implies that we have $\delta(C,x)<(p-1)/2$ for every closed point $x\in C$. 
Thus, $C$ is topologically rigid by Theorem \ref{MainTheorem1} and Corollary \ref{Corollary:K3}.
\end{proof}

\subsection{Nodal and cuspidal curves}
\label{Subsection:Example:NodalCuspidalRationalCurves}

Let $C$ be a reduced curve over $k$.
If $x \in C$ is a closed point, the $\delta$-invariant depends only on the completion $\widehat{\O}_{C,x}$.
Indeed, if $(\widehat{\O}_{C,x})'$ is the integral closure of $\widehat{\O}_{C,x}$ in its total ring of fractions, we have
$\delta(C, x)= \dim_k (\widehat{\O}_{C,x})'/\widehat{\O}_{C,x}$.
The Jacobian number also depends only on the completion;
see Proposition \ref{Proposition:JacobianNumberCompletion}
and Corollary \ref{Corollary:JacobianNumberDerivative}.
We leave the easy computations of the following result to the reader.

\begin{prop}
\label{Proposition:cusp}
Let $C$ be a reduced curve over $k$ and let $x\in C$ be a closed point.
\begin{enumerate}
\item  $x \in C$ is a \textit{node} (or \textit{ordinary double point})
if we have $\widehat{\O}_{C,x} \cong k[[ S,T ]]/(ST)$.
(See \cite[Chapter 7, Definition 5.13 and Proposition 5.15]{Liu:Book}.)
In this case, we have $\delta(C,x) = 1$ and $\jac(C, x)=1$.

\item $x \in C$ is an \textit{ordinary cusp}
if we have $\widehat{\O}_{C,x} \cong k[[ S,T ]]/(S^2+T^3)$.
The $\delta$-invariant is $\delta(C,x)=1$.
The Jacobian number depends on $p=\chara(k)$ as follows:
  \[
   \begin{array}{c||c|c|c|c}
             & p=0 & p=2 & p=3 & p\geq5 \\
   \hline
     \jac(C,x) & 2 & 4 & 3 & 2 \\
  \end{array}
  \]

\item  
Let $F_1, F_2, F_3 \in k[S,T]$ 
be distinct linear forms over $k$.
We put $F = F_1 \cdot F_2 \cdot F_3$.
If we have
$\widehat{\O}_{C,x} \cong k[[ S,T ]]/\left( F \right)$,
then we have $\delta(C, x)=3$ and $\jac(C,x) = 4$.
\end{enumerate}
\end{prop}

\begin{rem}
There are several equivalent definitions of ordinary cusp singularities, see 
for example, \cite[p.~308, Definition 2.17]{Shimada},
where basic properties of ordinary cusp singularities are studied
in arbitrary characteristic, including $p=2$ and $p=3$.
We also refer to \cite{GreuelKroning} for the classification of simple curve singularities in arbitrary characteristic $p$.
There, ordinary cusps arise as singularities of type $A_{2}$ (resp.\ $A^{0}_{2}$) when $p \neq 2$ (resp.\ $p=2$).
\end{rem}

Theorem \ref{MainTheorem1} implies the following.

\begin{cor}
\label{Corollary:RigidNodeCusp}
Let $X$ be a smooth, proper, and connected surface over $k$.
\begin{enumerate}
\item If $X$ contains a topologically non-rigid rational curve $C \subset X$ such that
every singularity of $C$ is a node,
then $X$ has negative Kodaira dimension.
\item If $p \geq 5$ and $X$ contains a topologically non-rigid rational curve $C \subset X$
such that every singularity of $C$ is a node or an ordinary cusp,
then $X$ has negative Kodaira dimension.
\end{enumerate}
\end{cor}

\begin{rem}
In characteristic $2$ or $3$, Corollary \ref{Corollary:RigidNodeCusp} (ii)
does not hold in general, because there exist quasi-elliptic surfaces 
of non-negative Kodaira dimension in these characteristics.
These contain topologically non-rigid cuspidal rational curves;
see Section \ref{Subsection:QuasiEllipticFibrations} below.
\end{rem}

\subsection{Topologically non-rigid rational curves and supersingular surfaces}

In positive characteristic, there exist surfaces of 
non-negative Kodaira dimension containing topologically non-rigid rational curves.
However, such surfaces have special properties:
if  $\rho(X)$ denotes the Picard number and if 
$b_2(X) := \dim_{\Q_{\ell}} H^2_{\text{\'et}}(X,\Q_{\ell})$
(which is independent of $\ell$ as long as $\ell\neq \chara(k)$) 
denotes the second Betti number, then 
\textit{Igusa's inequality} states that $\rho(X)\leq b_2(X)$; 
see \cite{Igusa}.
If $X$ contains a topologically non-rigid rational curve, then equality holds.

\begin{prop}
\label{Proposition:ExistenceNonRigidRationalCurvesPicardBetti}
 Let $X$ be a smooth, proper, and connected surface 
 over an algebraically closed field $k$ of characteristic $p \geq 0$.
 Assume that $X$ contains a topologically non-rigid rational curve $C \subset X$.
 Then, $X$ is uniruled and the equality
 $ \rho(X) = b_2(X)$ holds.
\end{prop}

\begin{proof}
The assertion follows from Shioda's results in \cite{Shioda} as follows.
First, $X$ is uniruled by Proposition \ref{Proposition:UniruledExistenceNonRigidRationalCurves}.
Then there exists a dominant rational map $\P^1 \times C \dashrightarrow X$
for a smooth and proper curve $C$ over $k$.
Since $\rho(\P^1 \times C) = b_2(\P^1 \times C) = 2$,
a theorem of Shioda \cite[Section 2, Lemma]{Shioda}
implies $\rho(X) = b_2(X)$.
(See also \cite[Proposition 14]{BogomolovHassettTschinkel}.)
\end{proof}

\begin{rem}
A surface $X$ in positive characteristic satisfying $\rho(X) = b_2(X)$ is 
called \textit{Shioda-supersingular}.
By Proposition \ref{Proposition:ExistenceNonRigidRationalCurvesPicardBetti},
every rational curve on a smooth, proper, and connected surface 
that is not Shioda-supersingular is topologically rigid;
see \cite[Section 9]{LiedtkeSurvey} for an overview of
supersingular surfaces.
\end{rem}

\begin{rem}
For K3 surfaces over an algebraically closed field $k$ of characteristic $p>0$,
there is another notion of supersingularity, which is due to Artin \cite{Artin:SupersingularK3}:
a K3 surface $X$ is called \textit{Artin-supersingular} if the height of its formal Brauer group 
$\widehat{\mathrm{Br}}(X)$ is infinite.
By the Tate conjecture for K3 surfaces
\cite{Charles, KimMadapusiPera, MadapusiPera,  Maulik, Nygaard, NygaardOgus},
these two notions are equivalent.
(In characteristic $2$, see also \cite{IIK2018}.)
See \cite[Section 17 and Corollary 17.3.7]{Huybrechts:K3Book} and \cite{LiedtkeSurvey2} for details.
\end{rem}

\subsection{Quasi-elliptic fibrations}
\label{Subsection:QuasiEllipticFibrations}

Arguably, the best-studied examples of families of non-smooth rational curves moving on 
a surface are \textit{quasi-elliptic} fibrations, i.e., fibrations
$X \to C$, whose geometric generic fiber is a rational curve with an ordinary cusp singularity.
Such fibrations can and do exist in characteristic $2$ or $3$.
We refer to \cite{Badescu, BombieriMumford2, BombieriMumford3} for results and a detailed analysis.
For details on quasi-elliptic fibrations in characteristic $3$,
we refer to \cite{CossecDolgachev, Lang}.

\begin{ex}
\label{Quasi-ellipticK3char=2}
Rudakov and \v{S}afarevi\v{c} showed that
every supersingular K3 surface $X$ over an algebraically closed field $k$
of characteristic $p$ such that
$p=2$, or $p=3$ and $\sigma_0(X)\leq6$
admits a quasi-elliptic fibration; see \cite[Theorem 1]{RudakovShafarevich:Char2}
and \cite[Section 5]{RudakovShafarevich:K3FiniteChar}.
Here, $\sigma_0$ denotes the \textit{Artin-invariant}
as introduced by Artin in \cite{Artin:SupersingularK3}.
\end{ex}

\begin{ex}
\label{ex:Quasi-ellipticK3char=3}
We use the notation from Example \ref{ex:Fermat} below:
if $n=4$ and $p=3$, then for a generic choice of $[s_0:s_1]$ the curve 
$C_{[s_0:s_1]}$  is a rational curve with an ordinary cusp on
the Fermat surface $S_4$ of degree $4$.
The $\delta$-invariant (resp.\ Jacobian number) of the singular point of $C_{[s_0:s_1]}$ is $1$ (resp.\ 3) by Proposition \ref{Proposition:cusp} (ii).
Moreover, the fibration $S_4 \to \P^1$ is quasi-elliptic and the line $\ell\subset S_4$
is a multisection of degree $3$.
We also note that $S_4$ is isomorphic to the unique supersingular K3 
surface with Artin-invariant $\sigma_0=1$ in characteristic $3$.
\end{ex}

\begin{ex}
 \label{ex:quasi-elliptic}
  Let $X \to C$ be a quasi-elliptic fibration with a section over an algebraically closed
  field $k$ of characteristic $p=3$.
  The possible types of the fibers were classified by Lang in \cite[p.~479, Section 1.B]{Lang}; see also \cite[Proposition 5.5.9]{CossecDolgachev}.
  There are four possibilities, where we use the notation
  in terms Kodaira-N\'eron types as well as the one from \cite{CossecDolgachev}.
  \begin{enumerate}
  \item (Type $II$, also denoted $\widetilde{A}_0^{\ast \ast}$).
    This is the generic type.
    There exists an open and dense subset $U\subseteq\P^1$,
    such that the fiber $f^{-1}(u)$ for every $u\in U$ is of this type.
    The fiber is a rational curve with one ordinary cusp,
    whose $\delta$-invariant (resp.\ Jacobian number) is $1$ (resp.\ $3$) by Proposition \ref{Proposition:cusp} (ii).

  \item (Type $IV$, also denoted by $\widetilde{A}_2^{\ast}$).
    The fiber is the union of three smooth rational curves intersecting
    at one point with three different tangent directions.
    The $\delta$-invariant (resp.\ Jacobian number) of the intersection point is $3$ (resp.\ $4$) by Proposition \ref{Proposition:cusp} (iii).

  \item (Type $IV^{\ast}$, also denoted by $\widetilde{E}_6$).
    The reduced part of the fiber is the union of seven smooth rational curves:
    three of them are reduced, three of them have multiplicity $2$,
    and one of them has multiplicity $3$.

  \item (Type $II^{\ast}$, also denoted by $\widetilde{E}_8$).
    The reduced part of the fiber is the union of nine smooth rational curves:
    one of them is reduced,
    two of them have multiplicity $2$,
    two of them have multiplicity $3$,
    two of them have multiplicity $4$,
    one of them have multiplicity $5$, and
    one of them has multiplicity $6$.
 \end{enumerate}
   In the degenerate cases (i.e., fibers of Type $IV$, $IV^{\ast}$, or $II^{\ast}$),
   the reduced components of the fibers are smooth rational curves.
   For a degenerate fiber of Type $IV^{\ast}$ or $II^{\ast}$,
   the reduced part of the fiber is a union of smooth rational curves
   with transversal intersection.
   This does not contradict our results because most of the rational curves
   in the fiber have multiplicities greater than $1$.
   On every point $x$ of a non-reduced curve $C$, we have $\jac(C, x)=\infty$.
   There are infinitely many such points on a degenerate fiber of Type $IV^{\ast}$ or $II^{\ast}$.
\end{ex}

\begin{rem}
Example \ref{ex:quasi-elliptic} illustrates the following:
\begin{enumerate}
 \item $\delta$-invariants and Jacobian numbers need not stay constant in a family of rational curves.
  For example, in a quasi-elliptic fibration in characteristic $3$,
  the general fiber $C$ has one point $x$ with $\jac(C, x)=3$.
  For a fiber $C$ of type $IV$, there is one point $x$ with $\jac(C, x)=4$.
  For a fiber $C$ of type $IV^{\ast}$ or $II^{\ast}$,
  we have $\jac(C, x)=\infty$ for infinitely many points $x$.

 \item  By Corollary \ref{Corollary:RigidNodeCusp},
  the union of two smooth rational curves meeting transversally on
  a smooth, proper, and connected surface of non-negative 
  Kodaira dimension is topologically rigid.
  But the union of three smooth rational curves intersecting at one point (such a curve is \'etale locally isomorphic to $XY(X+Y) = 0$)
  need not be, as fibers of type $IV$ of quasi-elliptic fibrations in characteristic $3$ show.
  (However, we note that if the intersection has three different tangent directions,
  then the intersection point has Jacobian number $4$
  in every characteristic and thus, this configuration would be topologically rigid on a surface of non-negative Kodaira dimension in characteristic $p\geq5$.)

 \item A configuration of curves, such that the reduced part is
  a transverse intersection of smooth rational curves,
  but which has components of multiplicity at least $2$,
  may not be topologically rigid on a smooth, proper, and connected surface of 
  non-negative Kodaira dimension:
  fibers of type $IV^{\ast}$ and $II^{\ast}$ in quasi-elliptic fibrations
  in characteristic $3$ provide examples.
\end{enumerate}
\end{rem}

\subsection{Topologically non-rigid rational curves with large $\delta$-invariants and Jacobian numbers}
\label{subsec:large Jacobian}

In this subsection, we will see that in characteristic $p \geq 3$, 
there exist surfaces of non-negative Kodaira dimension 
that contain topologically non-rigid rational curves that have precisely
one singular point, which is of $\delta$-invariant (resp.\ Jacobian number) 
equal to $(p-1)/2$ (resp.\ $p$).
Thus, Theorem \ref{MainTheorem1} is in some sense optimal.

We start with an auxiliary result, which shows that also Theorem \ref{MainTheorem3}
and Theorem \ref{Theorem:TateGenusChange}
are in some sense optimal.

\begin{lem}
\label{Lemma:QuasiHyperelliptic}
Let $k$ be a field of characteristic $p\geq3$.
Let $K:=k(t)$ be the field of rational functions over $k$ of the variable $t$.
Then, there exists a proper curve $C$ over $K$ 
satisfying the following three conditions:
\begin{enumerate}
\item $C$ has a unique singular point, whose geometric $\delta$-invariant (resp.\ Jacobian number) is $(p-1)/2$ (resp.\ $p$),
\item $C$ is a regular scheme, and
\item $C_{K^{\alg}} := C_K \otimes_K K^{\alg}$ is a rational curve over $K^{\alg}$.
\end{enumerate}
In particular, the bound of Theorem \ref{MainTheorem3} is optimal.
\end{lem}

\begin{proof}
We consider the two affine curves
\[ C_1 : Y^2 = X^p + t \quad \text{and} \quad C_2 : Y'^2=X' + tX'^{p+1} \]
over $K$.
From these, we obtain a curve $C$ over $K$ by gluing 
the two curves $C_1$ and $C_2$ via the isomorphism
\[ \{\, X \neq 0 \,\} \,\cap\, C_1 \,\stackrel{\cong}{\longrightarrow}\, \{\, X' \neq 0 \,\} \,\cap\, C_2 \]
that is defined by $X\mapsto 1/X'$ and $Y\mapsto Y'/X'^{(p+1)/2}$; 
see \cite[Proposition 7.4.24]{Liu:Book}.
Moreover, gluing the two morphisms
\[ C_1 \to \Spec K[X], \quad (X, Y) \mapsto X \]
and
\[ C_2 \to \Spec K[X'], \quad (X', Y') \mapsto X', \]
we obtain a finite morphism $C \to \P^1_{K}$.
In particular, the curve $C$ is proper over $K$.
The curve $C$ is regular, but it is not smooth over $K$;
see also \cite[Chapter II, Exercise 6.4]{Hartshorne}.
The closed point $x \in C_1$ corresponding to 
the maximal ideal $(X^p+t, Y) \subset K[X, Y]$
is the unique singular point of $C$.
It is easy to see that $C$ satisfies all the conditions of the lemma.
\end{proof}

We now construct a surface $Y$ of general type over $k$
that contains a topologically non-rigid rational curve, which has one singular 
point of $\delta$-invariant (resp.\ Jacobian number) equal to $(p-1)/2$ (resp.\ $p$).
In fact, these constructions are inspired by Raynaud's counterexamples 
to the Kodaira vanishing theorem in positive characteristic from 
\cite[Section 3.1]{Raynaud:Counter}.

\begin{prop}
\label{Proposition:JacobianOptimal}
Let $k$ be an algebraically closed field of characteristic $p\geq3$.
Then, there exists a smooth, projective, and connected surface $Y$ over $k$
satisfying the following conditions:
\begin{enumerate}
\item The Kodaira dimension of $Y$ satisfies $\kappa(Y)\geq1$. 
If $p\geq5$, then we may even assume $\kappa(Y)=2$, i.e.,
$Y$ is a surface of general type.
\item $Y$ contains a topologically non-rigid rational curve $C \subset Y$, and
\item $C$ has a unique singular point, whose $\delta$-invariant (resp.\ Jacobian number) is equal to $(p-1)/2$ (resp.\ $p$).
\end{enumerate}
\end{prop}

\begin{proof}
Let $C$ be the proper curve over $K=k(t)$ from Lemma \ref{Lemma:QuasiHyperelliptic}.
There exists a smooth, projective, and connected surface $X$ over $k$
together with a proper flat morphism $X \to \P^1$, whose generic fiber
satisfies $X \times_{\P^1} \Spec K \cong C$.
We choose a smooth and projective curve $S$ of genus $g(S) \geq 2$
and a generically \'etale morphism $S \to \P^1$.
Let $Y \to X \times_{\P^1} S$ be a resolution of singularities of 
$X \times_{\P^1} S$.
Then, the generic fiber of $Y \to S$ is isomorphic to $C_{k(S)}$. 
By \cite[Proposition 2.2]{Kollar:RationalCurveBook} and the proofs of 
Proposition \ref{Proposition:UpperSemicontinuityJacobianNumbers} and Proposition \ref{Proposition:UpperSemicontinuityDeltaInvariant}, there exists an open and dense subset $U \subset S$ such that
for every closed point $u\in U$,
the fiber $Y_u$ is a rational curve over $k$, and
the fiber $Y_u$ has a unique singular point and its $\delta$-invariant (resp.\ Jacobian number) is $(p-1)/2$ (resp.\ $p$).
Since $g(S)\geq2$, the Kodaira dimension of $S$ is equal to $\kappa(S)=1$.
If $p=3$, then the arithmetic genus of $C$ is equal to $p_a(C)=1$ and if $p\geq5$, then 
we even have $p_a(C)\geq2$.
Thus, we find $\kappa(Y)\geq1$ (resp.\ $\kappa(Y)=2$) if $p\geq3$ (resp.\ if $p\geq5$)
by a characteristic-$p$ version of Iitaka's $C_{1,1}$-conjecture;
see \cite[Theorem 1.3]{ChenZhang} for example.
\end{proof}

\subsection{Topologically non-rigid rational curves on Fermat surfaces}

In this subsection, we discuss 
topologically non-rigid rational curves on Fermat surfaces in characteristic $p > 0$.
Let $S_n \subset \P^3$ be the \textit{Fermat surface} of degree $n$ in $\P^3$
defined by
$X^n-Y^n + Z^n-W^n = 0$.
The signs of the defining equation are chosen so that it is easier 
to write down a line on $S_n$; see Example \ref{ex:Fermat} below.
The surface $S_n$ is smooth if and only if $p$ does not divide $n$, 
which we will assume from now on.
If $n \leq 3$, then $S_n$ is a rational surface,
$S_4$ is a K3 surface, and if
$n \geq 5$, then $S_n$ is a surface of general type.
In particular, if $n \geq 4$, then $S_n$ has non-negative Kodaira dimension.
On the other hand, Shioda showed that $S_n$ is unirational
if there exists an integer $\nu$ such that $p^\nu \equiv -1 \pmod{n}$; see \cite[Proposition 1]{Shioda}.

\begin{ex}
\label{ex:Fermat}
Assume $p \geq 3$. The Fermat surface $S_n \subset \P^3$ contains
the line $\ell$ defined by $X-Y = Z-W = 0$.
Let 
\[
  H_{[s_0:s_1]} \,=\, \{\, s_0 (X-Y) \,+\, s_1 (Z-W) \,=\, 0 \,\},
  \quad
  [s_0:s_1]\in\P^1_k
\]
be the pencil of planes in $\P^3$ containing the line $\ell$.
Then, $H_{[s_0:s_1]}\cap S_n$ is equal to the union of $\ell$ 
and a curve $C_{[s_0:s_1]}$, which is a plane curve of degree $(n-1)$
inside $H_{[s_0:s_1]}\cong\P^2$.
The above pencil gives rise to a fibration $S_n \to \P^1$,
whose fiber over $[s_0:s_1]\in\P^1$ is equal to $C_{[s_0:s_1]}$;
see also \cite[Exercise 7.4]{Badescu}.
Assume $n=p+1$.
For a generic choice of $[s_0:s_1]$, the curve 
$C_{[s_0:s_1]}$ is a rational curve, which has a unique intersection
point with $\ell$, in which the curve has a singularity.
The intersection multiplicity of $\ell$ with $C_{[s_0:s_1]}$ in this point is equal to $p$ and the line $\ell$ defines a degree-$p$ multisection of the fibration $S_n\to\P^1$.
The rational curve
$C_{[s_0:s_1]} \subset S_{p+1} \subset \P^3$
has its singular point at 
\[ [X:Y:Z:W] = [ s_0^{1/p} : s_0^{1/p} : s_1^{1/p} : s_1^{1/p} ]. \]
By an explicit calculation, which we omit, its $\delta$-invariant (resp.\ Jacobian number) is given by $(p-1)(p-2)/2$ (resp.\ $p(p-2)$).
\end{ex}

\begin{rem}
The above mentioned results of Shioda \cite{Shioda} on the unirationality of
Fermat surfaces have been generalized to Delsarte surfaces
by Katsura and Shioda \cite{ShiodaKatsura}.
In \cite{LiedtkeSchuett}, these have been used by Sch\"utt and 
the third named author to construct unirational surfaces on the Noether line for 
most values of $p_g$ and in most positive characteristics $p$.
\end{rem}

\subsection{Topologically non-rigid rational curves that are not contained in some fibration}

We will give examples of families of singular rational curves 
that are not contained in some fibration, whose geometric generic fiber is a rational curve.

\begin{lem}\label{PurelyInsepGroup}
Let $k$ be an algebraically closed field of characteristic $p>0$, and let $f : X \to Y$ be a dominant morphism
of smooth, proper, and connected surfaces over $k$, such that the induced extension
of function fields $k(Y)\subset k(X)$ is purely inseparable.
Then, the induced morphism 
$
\pi_1^{\rm{\acute{e}t}}(X)
\to \pi_1^{\rm{\acute{e}t}}(Y)
$
of \'etale fundamental groups is an isomorphism.
\end{lem}

\begin{proof}
There exists an open subset $V \subset Y$ such that 
$f^{-1}(V) \to V$ is finite and $Y \backslash V$ consists of finitely many closed points. 
The induced morphism  $\pi_1^{\rm{\acute{e}t}}(V)\to \pi_1^{\rm{\acute{e}t}}(Y)$
is an isomorphism by \cite[Corollaire X.3.3]{SGA 1}. 
The induced morphism 
$\pi_1^{\rm{\acute{e}t}}(f^{-1}(V))\to \pi_1^{\rm{\acute{e}t}}(V)$
is an isomorphism by Lemma \ref{Lemma:Radicial} and \cite[Th\'eor\`eme IX.4.10]{SGA 1}. 
Since $X$ is normal, the induced morphism
$\pi_1^{\rm{\acute{e}t}}(f^{-1}(V))\to \pi_1^{\rm{\acute{e}t}}(X)$
is surjective.
Therefore, the induced morphism
$\pi_1^{\rm{\acute{e}t}}(X)\to \pi_1^{\rm{\acute{e}t}}(Y)$
is an isomorphism.
\end{proof}


\begin{lem}
  Let $X$ be a smooth, proper, and connected surface over an algebraically closed field $k$
  and assume that the \'etale fundamental group $\pi_1^{\rm{\acute{e}t}}(X)$ is finite, but non-trivial.
  Then, neither $X$ nor any smooth, proper, and connected surface that is birationally equivalent to $X$, 
  admits a fibration, whose geometric generic fiber is a rational curve.
\end{lem}

\begin{proof}
Assume that there exists
a smooth, proper, and connected surface $Y$ that is birationally equivalent to $X$
and such that $Y$ admits a fibration $Y\to B$,
where $B$ is a smooth and proper connected curve, 
and whose geometric generic fiber is a singular rational curve.
As in the proof of (ii) $\Rightarrow$ (i) of Proposition \ref{Proposition:UniruledExistenceNonRigidRationalCurves}, 
there exists a purely inseparable covering $C\to B$
such that the normalization $Z$ of the induced fibration $Y\times_BC\to C$ is generically 
a $\P^1$-bundle. 
Let $\widetilde{Z}\to Z$ be a resolution of singularities.
The induced composition $f: \widetilde{Z}\to Y$ is dominant and 
induces an extension of function fields  $k(Y)\subset k(\widetilde{Z})$ 
that is finite and purely inseparable.
By Lemma \ref{PurelyInsepGroup}, we have
$\pi_1^{\rm{\acute{e}t}}(\widetilde{Z})\cong\pi_1^{\rm{\acute{e}t}}(Y)$.
By birational invariance of the fundamental group \cite[Corollaire X.3.4]{SGA 1},
we have $\pi_1^{\rm{\acute{e}t}}(Y)\cong\pi_1^{\rm{\acute{e}t}}(X)$. 
Therefore, we have
$\pi_1^{\rm{\acute{e}t}}(\widetilde{Z})\cong\pi_1^{\rm{\acute{e}t}}(X)$.
Hence, $\pi_1^{\rm{\acute{e}t}}(\widetilde{Z})$ is a finite and non-trivial group.
In particular, the first Betti number of $\widetilde{Z}$
satisfies $b_1(\widetilde{Z})=0$.
Passing to Albanese varieties and using their universal properties, we conclude 
$C\cong\P^1$.
Thus, $\widetilde{Z}$ is a rational surface.
This implies that $\pi_1^{\rm{\acute{e}t}}(\widetilde{Z})$ 
is trivial by \cite[Corollaire XI.1.2]{SGA 1}, a contradiction.
\end{proof}

\begin{rem}
 If $X\to Y$ is a finite \'etale morphism between smooth and projective varieties, 
 then $X$ is unirational if and only if $Y$ is unirational; see 
 \cite[Lemma 4]{ShiodaSupersingular}.
\end{rem}


\begin{ex}
 Let $S_n \subset \P^3$ be the Fermat surface of degree $n$ over an algebraically
closed field $k$ of characteristic $p\geq3$ with $p\nmid n$.
Let $m\geq4$ be a divisor of $n$ and let $\zeta=\zeta_m$ be a primitive $m$.th root of unity.
The $\mu_m$-action
$[X:Y:Z:W] \mapsto [X : \zeta Y : \zeta^2 Z : \zeta^3 W]$.
restricts to a fixed point free action of $\mu_m$ on $S_n$.
The quotient $Y_{n,m}:=S_n/\mu_m$ is a smooth, projective, and 
connected surface over $k$
with \'etale fundamental group $\pi_1^{\rm{\acute{e}t}}(Y_{n,m})\cong\Z/m\Z$.
Thus, if there exists an integer $\nu$ such that $p^\nu\equiv-1\pmod n$, then
\begin{enumerate}
 \item $Y_{n,m}$ contains a topologically non-rigid rational curve $C$, but
 \item $C$ is not the fiber of a fibration of $Y_{n,m}$, or of any smooth, proper, and connected
  surface birationally equivalent to $Y_{n,m}$, whose geometric generic fiber is a rational 
  curve.
\end{enumerate}
 For example, if $m=n=5$ and $p\neq 5$, then $Y_{5,5}$ is the classical \textit{Godeaux surface}.
This surface is unirational if and only if $p\not\equiv1\pmod5$; see \cite[Lemma 3]{Shioda}.
\end{ex}

\subsection{Higher dimensional counterexamples}
\label{subsec:higher dimensional counter-examples}

Finally, we give some examples that show that
a naive generalization of Theorem \ref{MainTheorem1}
(namely, omitting the separability condition (iii) in Theorem \ref{MainTheorem2}) 
to higher dimensions fails, even if we assume that the topologically non-rigid rational curve $C$ is smooth.
This is related to the following non-reducedness phenomenon in characteristic $p>0$:
if the target of a fibration between smooth varieties has dimension $\geq 2$,
then the geometric generic fiber may be non-reduced.
Such \textit{wild fibrations} and \textit{wild conic bundles} have been
constructed and studied, for example in \cite{Kollar:ExtremalRays, MoriSaito, Sato:Uniruled, Schroeer:Nagoya}.
Our examples of varieties given below were inspired by 
Sato's threefolds from  \cite{Sato:Uniruled}.
A special case in characteristic $2$ was also studied by Koll\'ar in
\cite[Example 4.12]{Kollar:ExtremalRays} and
\cite[Chapter IV, Exercise 1.13.5]{Kollar:RationalCurveBook}.

\begin{prop}
\label{Proposition:CounterexampleHigherDimension}
Let $k$ be an algebraically closed field of characteristic $p>0$ and let
$n\geq3$ be an integer.
Then,  there exist a smooth, projective, and connected variety $X$ 
of dimension $n$ over $k$,
a smooth and connected variety $U$ over $k$ with $\dim(U) = \dim(X) - 1$,
and a closed subvariety $\mathscr{C} \subset U \times X$
(with projections $\pi \colon \mathscr{C} \to U$ and $\varphi \colon \mathscr{C} \to X$)
satisfying the following conditions:
\begin{enumerate}
\item $X$ is \textit{not} separably uniruled,
\item $\mathscr{C} \subset U \times X$ gives a family of rational curves on $X$
(see Definition \ref{Definition:FamilyRationalCurve} (iii)),
\item $\varphi \colon \mathscr{C} \to X$ is dominant, and
\item for every closed point $u \in U$, $\varphi(\mathscr{C}_u)$ is a smooth rational curve on $X$.
\end{enumerate}
\end{prop}

\begin{proof}
First, we consider the case $n=3$.
In this case, the existence such a smooth and projective threefold $X$
follows from work of Sato \cite{Sato:Uniruled}:
let $X$ be a three-dimensional example as stated
in \cite[p.~448, Theorem]{Sato:Uniruled} (the construction is explained in \cite[p.~458, at the beginning of Section 5]{Sato:Uniruled}).
In our situation, we can discard the condition $p<(n+3)/2$ from \cite[p.~448, Theorem]{Sato:Uniruled},
since this is necessary for $X$ to also satisfy the property $\mathrm{(NC)}$
(introduced in \cite[p.~447]{Sato:Uniruled} and proved in \cite[p.~460, Step 2]{Sato:Uniruled}), which we do not require.
A family of rational curves $\mathscr{C}\subset U\times X$, such that $\varphi(\mathscr{C}_u)$ 
is a smooth rational curve on $X$ for every closed point $u\in U$ is provided by the
lines $L$ from \cite[p.~460, Step 2]{Sato:Uniruled}.

If $n\geq4$, then we take a threefold $X$ as above.
We choose a smooth, projective, and connected curve $C$ of genus $g\geq1$
over $k$.
By Lemma \ref{Lemma:SeparablyUniruledFiberProduct}, the product
\[ X \times \underbrace{C \times \cdots \times C}_{(n-3)\,\rm{factors}} \]
yields an $n$-dimensional example that
satisfies the properties of Proposition \ref{Proposition:CounterexampleHigherDimension}.
\end{proof}

\begin{rem}
From the construction given in \cite{Sato:Uniruled}, one easily sees that
$k(\mathscr{C}) \cap k(X)^{\sep}$ is \textit{not} a separable extension of
$k(U) \cap k(X)^{\sep}$, i.e., condition (iii) of Theorem \ref{MainTheorem2} is not satisfied.
\end{rem}

\subsection*{Acknowledgements}

The authors thank Frank Gounelas and Ichiro Shimada
for comments and discussion.
The authors thank Hiromu Tanaka for
providing invaluable information on the geometry of surfaces
over imperfect fields and references.
Moreover, the authors thank Gert-Martin Greuel for many comments and suggestions, 
including a whole report on an earlier version of this article.
Finally, the authors thank the referee for remarks and comments, which improve the article.

The first named author is supported by Research Fellowships of Japan Society for the Promotion of Science for Young Scientists KAKENHI Grant Number 18J22191.
The second named author is
supported by the JSPS KAKENHI Grant Number 20674001 and 26800013.
The third named author is supported by the ERC Consolidator Grant
681838 K3CRYSTAL.

\end{document}